\documentclass[10pt,reqno]{amsart}  %{iopart}
\usepackage{amsmath,amsfonts,amssymb,graphicx,amscd,amsthm,amsbsy,epsf}
\usepackage{comment}
\usepackage[ps2pdf=true]{hyperref}
\newtheorem{Lemma}{Lemma}[section]
\newtheorem{Prop}[Lemma]{Proposition}
\newtheorem{Thm}[Lemma]{Theorem}

\newcommand{\Real}{\mbox{${\mathbb R}$}}

\newcommand{\Sfrac}[2]{\mbox{\small$\frac{#1}{#2}$\normalsize}}            
\newcommand{\Half}{\Sfrac{1}{2}}
\newcommand{\sech}{\mbox{$\mathrm{sech}$}}

\newcommand{\BigOh}[1]{\mathcal{O}(#1)}
\newcommand{\LittleOh}[1]{\it{o}(#1)}

\begin{document}
\title[Water waves  over random topography]
{Long wave expansions for water waves  over random topography}
\author[A. de Bouard, W. Craig, O. D\'{\i}az--Espinosa, P. Guyenne,
C. Sulem]{Anne de Bouard$^1$,
 Walter Craig$^2$,
 Oliver  D\'{\i}az-Espinosa$^2$,
 Philippe Guyenne$^3$
 and  Catherine Sulem$^4$}

\address{$^1$Centre de Math\'ematiques Appliqu\'ees, Ecole Polytechnique,
 Route de Saclay, 91128 Palaiseau Cedex France}
\address{$^2$ Department of Mathematics and Statistics,  McMaster University,
 1280 Main St. West, Hamilton, ON L8S4K1, Canada}
\address{$^3$ Department of Mathematics, University of Delaware, 501 Ewing Hall
 Newark, DE 19716-2553,  USA}
\address{$^4$ Department of Mathematics,  University of Toronto, 
40 St George St, Toronto, ON M5S2E4, Canada}
\begin{abstract}
In this paper, we study the motion of the 
free surface of a body of fluid  over a variable bottom,
in a long wave asymptotic regime. We assume that the
 bottom of the fluid region can be  described by a stationary random 
process $\beta(x, \omega)$ whose variations take place on  short 
length scales and which are decorrelated on the length scale of 
the long waves. This is a  question of homogenization theory in 
the scaling regime for the Boussinesq and KdV equations.
 
The analysis is performed from the point of view of perturbation 
theory for Hamiltonian PDEs with a small parameter,  in the context
of which we perform a careful analysis of the distributional 
convergence of stationary mixing random processes.
We show in particular that the problem does not fully homogenize,
and that the random effects are as important as dispersive and
nonlinear phenomena in the scaling regime that is studied.
Our principal result is the derivation of effective equations for
surface water waves in the long wave small amplitude regime, and  
a consistency analysis of  these equations, which are not 
necessarily Hamiltonian PDEs. In this analysis we compute the 
effects of random modulation of solutions, and give an explicit 
expression for the scattered component of the solution
due to waves interacting with the random bottom. We  show
that the resulting influence of the random  topography is 
expressed in terms of a canonical process, which is equivalent to a
white noise through Donsker's invariance principle, with one
free parameter being the variance of the random process $\beta$.
This work is a reappraisal of the paper by Rosales \& Papanicolaou~\cite{RP83}
and its extension to general stationary mixing processes.
\end{abstract}
\bigskip\bigskip\bigskip
%Uncomment for PACS numbers title message
\noindent
%\qquad \qquad {\small{Keywords :{\it  Water waves,
% random  topography, long wave asymptotics}}}
% Uncomment for Submitted to journal title message
%\submitto{\JPA}

\maketitle

%%%%%%%%%%%%%%%%%%%%%%%%%%%%%%%%%%%%%%%%%%%%%%%%%%%%%%%%%%%%%%%%%%%%%%
%%                                                                  %%
%%  Section: Introduction                                           %%
%%                                                                  %%
%%%%%%%%%%%%%%%%%%%%%%%%%%%%%%%%%%%%%%%%%%%%%%%%%%%%%%%%%%%%%%%%%%%%%%

\section{Introduction}
\label{Section:Introduction}

The problem of surface water waves over an uneven bottom
 is a classical problem of fluid mechanics, and it is relevant
 to coastal engineering and ocean wave dynamics.   
In this paper, we investigate how the presence of bottom topography 
affects the  equations describing the limit of solutions in the
 long wave regime. We assume that the bottom is modeled by a stationary 
random process which is mixing, whose variations and  whose 
correlation length manifest themselves on  length scales that are 
short  compared to
the scale of the surface waves. 
 In a previous work \cite{CGNS05}, we addressed the long wave limit
 of surface waves over a bottom which has periodic variations over 
short scales, in which we proved that the problem fully homogenizes.
That is to say, 
 the free surface motion can be described by
a partial differential equation with constant {\it effective}    
coefficients, where the dependency over short scales
is manifested by coefficients which are ensemble averages.
 Here in contrast, we show that {\it random,
realization-dependent} effects are retained in the  description of the 
solution. The latter paper and
the present one are reappraisals and extensions of an earlier 
work by Rosales \& Papanicolaou \cite{RP83} who address the problem through 
different methods.

Our approach uses  a formulation in terms of  perturbation
theory for Hamiltonian partial differential equations, coupled
with a detailed analysis of stationary ergodic processes which
have mixing properties and which are considered as
tempered distributions.  As a first result we give an
appropriate form of the Boussinesq equations. Secondly,
following a series of  changes of variables, we derive a
system of  coupled KdV-like equations for the
two components of the solution;  these describe a
wave propagating predominantly to the right,
and a `small' scattered  wave propagating to the left.
We then extract a limiting system of two effective
equations through a consistency analysis. Specifically,
we solve the effective system, which is composed of an
equation similar to the KdV for the wave propagating to the right
with a random component to its velocity, and a scattered wave
propagating to the left. We give explicit formulas for the dominant
contributions and the first corrections to this solution, quantifying
the effects of the random modulation of position and amplitude.
 From these expressions, we compute {\it a posteriori} all the terms
that have been neglected in the effective system, and prove that
they are indeed of higher order. This evaluation relies on scale
separation lemmas, which in turn follow from Donsker's invariance
principle. Our analysis improves upon \cite{RP83} in several ways.
In particular we identify the canonical limiting distributions which
contribute to the random asymptotic behavior of solutions, we
quantify both random phase and random amplitude variations
of solutions, and in addition, we extend the long-wave analysis
over random topography to general stationary mixing processes.

The asymptotic system of equations that results from this analysis 
consists of  a KdV equation with an additional linear term, and a transport
equation for the scattered component driven by an  inhomogeneous
forcing term. The additional nonzero linear term,
which  either stabilizes or destabilizes solutions
depending upon the sign of its coefficient, in turn depends on the 
statistics of the bottom variations. The presence of this term 
 is the consequence of a subtle
calculation, and to our knowledge, it has not been previously observed.
In case these statistics are spatially reversible, the relevant 
 coefficient vanishes and the 
equation reduces to the usual KdV. 

There has been a lot of interest in wave motion in basins with 
non constant bathymetry, due to its hydrodynamic importance. Recent
references to the theory of linear waves include the papers of 
Nachbin (1995) \cite{N95}, S\o lna \& Papanicolaou (2000)~\cite{SP00}, 
Nachbin \& S\o lna (2003)~ \cite{NS03} which discuss the theory 
of linear transport in a random medium.
The earlier work of Howe (1971) \cite{H71} and the paper of 
Rosales \& Papanicolaou (1983) \cite{RP83} give an asymptotic analysis of 
nonlinear equations of water waves.
Nonlinear problems over variable topography are addressed in 
 Nachbin (2003) \cite{N03} and Artiles \& Nachbin (2004) \cite{AN04}.
More recent 
contributions which take into account the combined effect of randomness
and nonlinearity  include the series of papers by Mei \& Hancock (2003)
 \cite{MH03}
and Grataloup and Mei (2003) \cite{GM03} on the modulational scaling regime, 
and its extensions to the three dimensional case in 
Pihl,  Mei \& Hancock \ (2002)~\cite{PMH02}. 
This work focuses on the temporal behavior of  ensemble averages 
of solutions, giving the result  that they satisfy a nonlinear
Schr\"odinger equation with an additional dissipative term. 
The analog of this picture in the 
long wave scaling regime appears in Mei \& Li (2004) 
 \cite{ML04}, where the bottom
is assumed random but varies on the same spatial scale as the surface waves.

There is a history of rigorous analysis of the initial value problem
and limiting equations in the long wave asymptotic regime of the
water wave problem. Most of  this work concerns the case of
fluid domains with a flat bottom. The papers that address the
KdV limit include Kano \& Nishida (1986) \cite{KN86},
Craig (1985) \cite{C85}, Schneider \& Wayne (2000)
\cite{SW00}, Wright (2005) \cite{W05}
and Bona, Colin \& Lannes (2005) \cite{BCL05}.  A recent
paper which addresses specifically the Boussinesq scaling limit
of the problem on a rigorous basis, and categorizes the
well-posed possible limits is Bona, Chen \& Saut  (2002) \cite{BCS02}.
There has been several papers giving a rigorous analysis of the
initial value problem of water waves over a variable bottom,
including Yosihara (1983) \cite{Y83} on the two-dimensional
problem and Alvarez-Samaniego \& Lannes (2006) \cite{AL07}
on the two and three-dimensional problems, and a recent paper by
Chazel (2007) \cite{C07}. The  paper  \cite{AL07}
considers the issue of convergence in various scaling regimes
governed by long wave models. These results are in the context
of a deterministic problem, with a small amplitude bottom
perturbation, varying spatially on the same scale as the
waves in the surface. As far as we know, there are no current
rigorous analytic results for the KdV or Boussinesq scaling
regimes in which the  bottom variations occur on a short
length scale, and are averaged under the nonlinear evolution
of water waves.

The paper is organized as follows. 
Section~2 describes the problem of water waves in its 
Hamiltonian form, the Dirichlet-Neumann operator in the presence
of a variable bottom, and the spatial scaling
regime appropriate for the long wave problem.
Section~3 presents
the setting of stationary ergodic and mixing processes in which we
work,  and gives the relevant scale separation lemmas.
This is the key  of the paper.
It furthermore gives an analysis of the natural 
regularization of characteristic coordinates that are applied 
to the KdV scaling limit. The Boussinesq regime is presented
in Section~4, while the  more detailed KdV regime is taken upon in Section~5.
The main issue of this analysis  is that the  scattering of waves by
the bottom variations is strong and it must be shown that the 
standard KdV Ansatz of unidirectional
propagation remains valid despite this. The consistency analysis of the 
resulting asymptotic system of equation is the most detailed part of this
paper. Finally, Section 6 presents some remarks  
on the process of  ensemble averaging.

%%%%%%%%%%%%%%%%%%%%%%%%%%%%%%%%%%%%%%%%%%%%%%%%%
%%
%%
%%
%% Section: Hamiltonian formulation of the water waves problem
%%
%%
%%
%%%%%%%%%%%%%%%%%%%%%%%%%%%%%%%%%%%%%%%%%%%%%%%%%%

\section{ Hamiltonian formulation}
\setcounter{equation}{0}
\subsection{Hamilton equations}
The time-dependent fluid domain consists of the region 
$S(\beta, \eta) = \{ (x,y)\in \Real^{n-1}\times\Real: 
-h + \beta(x) < y < \eta(x,t) \}$, in which
the fluid velocity is represented by the gradient of a velocity potential,
\begin{equation}\label{Eq:Laplace'sEquation}
   u = \nabla \varphi \, , \qquad \Delta \varphi = 0 .
\end{equation} 
The dependent variable  $\eta(x,t)$ denotes the surface elevation, 
and $\beta(x)$ 
denotes the variation of the bottom of the fluid domain from its 
mean value. The bottom  variations are chosen from a statistical ensemble
$(\Omega, \mathcal{M}, \sf{P})$, which is indicated by the notation 
$\beta = \beta(x, \omega)$. The details of the ensemble and the
associated probabilistic properties are described in 
Section~\ref{Subsection:StationaryProcesses}.

On the bottom boundary $\{ y = -h + \beta(x) \}$, the velocity potential
obeys Neumann boundary conditions
\begin{equation}\label{Eq:NeumannBottom}
   \nabla \varphi \cdot N(\beta) = 0
\end{equation}
where $N(\beta) = (1 + |\partial_x\beta|^2)^{-1/2} (\partial_x \beta, -1)$
is the exterior unit normal. 

The top boundary conditions are the usual
kinematic and Bernoulli conditions imposed on $\{ (x,y) : y = \eta(x,t) \}$, 
namely 
\begin{equation}\label{Eq:Kinematic+Bernoulli}
\partial_t \eta  =  
        \partial_y \varphi - \partial_x \eta \cdot \partial_x\varphi , 
     \; \;\;\;
  \partial_t \varphi  =  -g\eta -\Half |\nabla\varphi|^2 ~. 
\end{equation}

The asymptotic analysis in this paper is initiated  from the point of view of 
the perturbation theory of a Hamiltonian system with respect to a 
small parameter. For this purpose we describe the
water wave problem  as a Hamiltonian system with infinitely many degrees 
of freedom.
In \cite{Z68}, Zakharov  poses the equations 
of evolution
 \eqref{Eq:Laplace'sEquation}\eqref{Eq:NeumannBottom}\eqref{Eq:Kinematic+Bernoulli} 
in the form of a Hamiltonian system in the canonical variables 
$(\eta(x), \xi(x))$ where one defines
$\xi(x) = \varphi(x, \eta(x))$, the boundary values of the velocity potential 
on the free surface. 
 The evolution equations take the classical form
\begin{equation}\label{Eq:HamiltonsEquations}
   \partial_t  \begin{pmatrix} \eta \\ \xi 
               \end{pmatrix} 
   = \begin{pmatrix} 0 & I \\ -I & 0
     \end{pmatrix}
     \begin{pmatrix} \delta_\eta H  \\ \delta_\xi H 
     \end{pmatrix}
   = J \, \delta H
\end{equation}
with the Hamiltonian functional given by the expression
of the total energy
\begin{eqnarray}
   H & = & \int \int_{-h+\beta(x)}^{\eta(x)} 
         \frac{1}{2} |\nabla \varphi(x,y)|^2 \, dy dx 
       + \int \frac{g}{2} \eta^2(x) \, dx \nonumber  \\
     & = & \int \frac{1}{2} \xi(x) G(\beta, \eta) \xi(x) \, dx 
       + \int \frac{g}{2} \eta^2(x) \, dx ~.
\end{eqnarray}

The Dirichlet-Neumann operator $G(\beta, \eta)$ is the singular integral
operator with  which one  expresses the normal derivative of the velocity potential
on the free surface. It is a function of the boundary values $\xi(x)$ and of
the domain itself, as parameterized by  $\beta(x)$ and
$\eta(x)$ ,which define respectively the lower and the upper boundaries of the fluid
domain $S(\beta, \eta)$. That is, let $\varphi(x,y)$ satisfy the boundary value problem
\begin{eqnarray}\label{Eq:LaplacesEquationInS}
   \Delta \varphi & = & 0 \ \ \hbox{\rm in} \quad S(\beta, \eta) ~, \\
   \nabla \varphi \cdot N(\beta) & = & 0 \ \ \hbox{\rm on the bottom boundary} 
     \quad \{ y = -h + \beta(x) \} ~, \nonumber   \\ 
   \varphi(x, \eta(x)) & = & \xi(x) \ \ \hbox{\rm on the free surface} \quad
     \{ y = \eta(x) \} ~.  \nonumber 
\end{eqnarray}
The Dirichlet-Neumann operator is expressed as follows
\begin{equation}\label{Eq:DirNeuOp}
   G(\beta, \eta)\xi(x) 
     = \nabla \varphi(x,\eta(x)) \cdot N(\eta) 
       (1 + |\partial_x\eta|^2)^{1/2},
\end{equation}
where $N(\eta)$ is the exterior unit normal on the free surface.
It is clearly a linear operator in $\xi$ and it is self-adjoint with this
normalization.  However it is nonlinear with 
explicitly nonlocal behavior in  $\beta(x)$ and 
$\eta(x)$.
 The form of this operator, and its description in terms
of  $\beta$ and $\eta$ are given in the next section.

%%%%%%%%%%%%%%%%%%%%%%%%%%%%%%%%%%%%%%%%%%%%%%%%%%%%%%
%%
%%   Subsection: Description of the Dirichlet -- Neumann operator
%%
%%
%%
%%%%%%%%%%%%%%%%%%%%%%%%%%%%%%%%%%%%%%%%%%%%%%%%%%%%%%

\subsection{Description of $G(\beta, \eta)$}

We now restrict consideration to the dimension $n=2$.
In the undisturbed case in which the bottom is flat, 
 the solution is formally given by a 
Fourier multiplier operator in the $x$-variable. Using the notation 
that $\partial_x = iD$; 

\begin{equation}\label{Eq:FourMult}
   \varphi(x,y) =  \int\int e^{i k  (x-x')}
      \frac{\cosh(k(y+h))}{\cosh(kh)} \xi(x') \, dx' dk 
   =  \frac{\cosh((y+h)D)}{\cosh(hD)} \xi(x) ~. 
\end{equation}
When the bottom topography is nontrivial, as represented by
$\{ y = -h + \beta(x) \}$, the expression \eqref{Eq:FourMult} 
is modified by adding a second term in order that the solution 
satisfies the bottom boundary conditions
\begin{equation}\label{Eq:HarmonicFunctionsInS}
   \varphi(x,y) = \frac{\cosh((y+h)D)}{\cosh(hD)} \xi(x) 
    + \sinh(yD) (L(\beta) \xi )(x) ~.
\end{equation}
The first term in \eqref{Eq:HarmonicFunctionsInS} satisfies the homogeneous
Neumann condition at $y = -h$ while the second term satisfies the homogeneous
Dirichlet condition at $y = 0$.
The operator $L(\beta)$ in the second term acts on the 
boundary data $\xi(x)$ given on the free surface. In \cite{CGNS05} 
we analyzed  $L(\beta)$ in a nonperturbative 
case, where $|\beta|_{C^1} \sim \BigOh{1}$. Here we are restricted
to a perturbative regime, where we describe the expansion of the 
operator $G(\beta, \eta)$ for small but arbitrary perturbations 
$\eta(x)$ of the surface, and small bottom variations $\beta(x)$.

At order $\BigOh{1}$ and $\BigOh{\eta}$, one gets 
$G^{(0)}  = D\tanh (hD) +D L(\beta)$ and 
$G^{(1)} =  D \eta D - G^{(0)} \eta G^{(0)}$. 
At higher order, one finds  the same 
recursion formula for $G^{(l)}$ as for the case of a flat bottom \cite{CS93} 
except that the role of the operator $G_0 = D \tanh(hD)$ is now 
replaced by $G^{(0)}$. 

Since we allow bottom perturbations to be of order $\BigOh{\varepsilon}$,
we will use  a recursion formula given in \cite{CGNS05}
for $L(\beta)$ in powers of $\beta$.
\begin{equation}
L(\beta)=L_1(\beta)+ L_2(\beta)+....
\end{equation}
with the first terms being
\begin{eqnarray}
L_1(\beta)&=& -\sech(hD) \beta \sech(hD) D 
\label{L1}\\
L_2(\beta) &=& \sech(hD) \beta D \sinh(hD) L_1 \cr
&=& -\sech(hD) \beta D \tanh(hD)\beta D \sech(hD).
\label{L2}
\end{eqnarray}
General formulas are presented in \cite{CGNS05} together with 
 a Taylor expansion of the Dirichlet-Neumann operator $G(\beta,\eta)$
in powers of both $\beta$ and $\eta$. In the analysis of the present 
paper, we will need  only the
terms up to second order in $\beta$.

The Hamiltonian is thus expanded in powers of $\eta$ and $\beta$ in
the form
\begin{eqnarray}
   H(\eta,\xi;\beta) = \frac{1}{2} \int (\xi D \tanh(hD) \xi + g \eta^2) dx \cr 
 \quad
  - \frac{1}{2} \int \xi D \sech(hD) \beta D \sech(hD) \xi dx \cr
  \quad
   + \frac{1}{2} \int  \xi (D\eta D -D \tanh(hD) \eta D \tanh(hD) ) \xi dx \cr 
  \quad
  -\frac{1}{2} \int \xi( D\sech(hD) \beta D \tanh(hD)\beta D \sech(hD) ) \xi dx
    \cr
\quad
    + \BigOh{\beta^3 \xi^2 } + \BigOh{\eta \beta \xi^2} +
    \BigOh{\eta^2 \xi^2} ~.
\end{eqnarray}
By integration by parts, 
\begin{eqnarray}
  && H(\eta,\xi;\beta) 
    =  \frac{1}{2} \int (\xi D \tanh(hD) \xi + g \eta^2) dx  
    - \frac{1}{2} \int  \beta |D \sech(hD)\xi|^2   dx \cr
  && \quad
    + \frac{1}{2} \int  \xi (D\eta D -D \tanh(hD) \eta D \tanh(hD) ) \xi dx \cr 
  && \quad
    -\frac{1}{2} \int \overline{(D\sech(hD) \xi )} \beta D \tanh(hD)\beta D \sech(hD) \xi dx
    \cr
  && \quad
    + \BigOh{\beta^3 \xi^2} + \BigOh{\eta \beta \xi^2} + \BigOh{\eta^2 \xi^2}~,
\end{eqnarray}
which is the starting point for our asymptotic expansion.
%%%%%%%%%%%%%%%%%%%%%%%%%%%%%%%%%%%%%%%%%%%%%%%%%%%%%%%
%%
%%   Subsection: Spatial scaling and the scaled Hamiltonian
%%
%%
%%
%%%%%%%%%%%%%%%%%%%%%%%%%%%%%%%%%%%%%%%%%%%%%%%%%%%%%%%

\subsection{Spatial scaling and the scaled Hamiltonian}

We consider the case in which the bottom varies on a short 
length scale, that is $\beta = \beta(x,\omega)$ is a random 
process, of zero mean value that satisfies  ergodicity and 
mixing properties which will be detailed below. 

The fundamental long wave scaling for the problem of surface water 
waves retains a balance between linear dispersive and nonlinear 
effects in the dynamics of the surface evolution. The  
scaling that anticipates this balance is through the transformation 
\begin{equation}\label{BouScal}
   X = \varepsilon x, \quad  \xi(x)= \varepsilon \tilde\xi(X), \quad 
 \eta(x)=  \varepsilon^2 \tilde\eta(X).
\end{equation}
As for the bottom, we assume its variations 
are of order $\BigOh{\varepsilon}$, which are much larger that the variations
of the surface elevation, namely
\begin{equation}
  \beta(x,\omega) = \varepsilon \tilde\beta(x,\omega).
\end{equation}
We assume that $\tilde \beta$ is bounded in $C^1$ for almost every realization
$\omega \in \Omega$.

In order to get the scaled Hamiltonian, we need to examine
the asymptotic expansion  of the Dirichlet-Neumann operator $G(\beta, \eta)$
in a multiple scale regime.
We  recall how formally a pseudo-differential operator  acts
on a multiple scale function $f(x,X)$ where $X= \varepsilon x$
(see \cite{CSS92} for details). In particular 
let $m(D)$ be a Fourier multiplier operator acting on a function $f$,
defined as 
\begin{equation}
(m(D) f)(x) =  \frac{1}{2\pi} \int e^{ik(x-y)} m(k) f(y) dy dk.
\end{equation}
When $m(D)$ acts on a multiple scale function $f(x,X)$ with $X=\varepsilon x$,
$D$ is replaced by $D_x +\varepsilon D_X$ and 
\begin{eqnarray}
m(D) f(x,X) &=&  \frac{1}{2\pi}  \int e^{ik(x-y)}
\Big(\sum_{j=0}^\infty \frac{m^{(j)}(k)}{j!}\varepsilon^j D^j_X \Big)
 f(y,X) dy dk\cr
&=& m(D_x) f + \varepsilon m'(D_x) D_X f +\cdots 
\end{eqnarray}
Applying this to the scaled Hamiltonian, we get 
\begin{eqnarray}
H(\tilde\eta,\tilde\xi; \tilde\beta, \varepsilon) 
 = 
\frac{\varepsilon^3}{2} \int ( h \tilde\xi D_X^2 \tilde\xi +g \tilde\eta^2)
  dX  \label{schamil} \\
- \frac{\varepsilon^4}{2} \int \tilde\beta(x) 
| D_X \sech(\varepsilon h D_X)\tilde\xi|^2  dX  
+\frac{\varepsilon^5}{2}\int 
 \tilde\xi (D_X \tilde\eta D_X \tilde\xi-
 \frac{h^3}{3} D_X^4 \tilde\xi ) 
dX  \cr
- \frac{\varepsilon^5}{2}\int \overline{
(D_X\sech(\varepsilon h D_X) \tilde\xi)} 
\cr
 \quad 
\left[\tilde\beta(x)
 (D_x+\varepsilon D_X)\tanh(h (D_x+\varepsilon D_X))\tilde\beta(x)
 D_X \sech(\varepsilon h D_X) \tilde\xi\right] dX \cr
 +o(\varepsilon^5)  \nonumber 
\end{eqnarray}
For simplicity of notation, we now drop the tildes over $\beta, \eta, \xi$.
Expanding the operator $\sech(\varepsilon hD_X)$ in the
 second term in (\ref{schamil}) gives
\begin{equation}
 \int \beta(x) 
~| D_X \sech(\varepsilon h D_X)\xi|^2  dX 
= \int \beta(\frac{X}{\varepsilon})  
 ~\left| D_X(1 - \frac{1}{2}\varepsilon^2 h^2 D_X^2)\xi\right|^2 dX ~.
\end{equation}
The last term of (\ref{schamil}) is a little more complicated but 
is calculated in the same manner.
Expanding $(D_x+\varepsilon D_X) \tanh(h(D_x+\varepsilon D_X))$
we get
\begin{equation}
  (D_x+\varepsilon D_X) \tanh(h(D_x+\varepsilon D_X)) 
    = D_x\tanh(h D_x)  + {\mathcal O}(\varepsilon)~.
\end{equation}
Finally, 
\begin{equation}
\begin{array}{l}
\int \overline{D_X \sech(\varepsilon h D_X) \xi} 
\left[\beta(x)
 (D_x+\varepsilon D_X)\tanh(h (D_x+\varepsilon D_X))\beta(x)
 D_X \sech(\varepsilon h D_X) \xi\right] dX   \\
\noalign{\vskip7pt}
\quad= \int \overline{D_X\sech(\varepsilon h D_X) \xi} 
[\beta(x) D_x \tanh(hD_x)  \beta(x) ] D_X\sech(\varepsilon h D_X) \xi
dX +{\mathcal O}(\varepsilon)  \\
\noalign{\vskip7pt}
\quad = \int [\beta(x) D_x \tanh(hD_x)  \beta(x)] ~|D_X \xi|^2 dX 
+{\mathcal O}(\varepsilon) ~.
\end{array}
\end{equation} 
Putting all these terms together:
\begin{equation}
\begin{array}{l}
%\begin{eqnarray}
H(\eta,\xi;\beta,\varepsilon)  = 
 \displaystyle{   \frac{\varepsilon^3}{2} \int} \Big[\Big(
 h -\varepsilon \beta(x)-\varepsilon^2 
\beta(x) D_x\tanh(hD_x)\beta(x) \Big) |D_X\xi|^2  +g \eta^2 
\\
\noalign{\vskip7pt}
%\cr
%    - \frac{\varepsilon^4}{2} \int  \beta(x) |D_X\xi|^2 dX \cr
\quad  +\frac{\varepsilon^2}{2} (\xi D_X \eta D_X \xi- 
    \frac{h^3}{3}\xi D_X^4 \xi) \Big] dX  +  o(\varepsilon^5) ~.
%\quad -\frac{\varepsilon^5}{2}\int \Big(\beta(x) D_x\tanh(hD_x)\beta(x)\Big) 
%    | D_X\xi|^2 dX
%  +  o(\varepsilon^5) ~.
\label{Eq:scaledhamil}
\end{array}
\end{equation} 
%\end{eqnarray}
%%%%%%%%%%%%%%%%%%%%%%%%%%%%%%%%%%%%%%%%%%%%%%%%%
%%
%%   Section: Homogenization and scale separation
%%
%%
%%%%%%%%%%%%%%%%%%%%%%%%%%%%%%%%%%%%%%%%%%%%%%%%
\section{Homogenization and scale separation}\label{homog}
\setcounter{equation}{0}
The purpose of this section is to understand the asymptotic
behavior of integrals of the form
\begin{equation}\label{Eq:BasicMultiScaleIntegral}
   \int_{-\infty}^{+\infty}  \gamma(\frac{X}{\varepsilon}) f(X) \, dX  
   := Z_\varepsilon(\gamma, f) ~,
\end{equation}
where $f(X)$ comes from expressions which involve the physical
variables which depend only upon large spatial scales, and where
$\gamma(x) = \gamma(x; \omega)$ is a stationary ergodic process taken from 
 the statistical ensemble $\Omega$ from which our
realizations of the bottom are sampled. Principle examples of 
such integral expressions in the Hamiltonian for water waves are 
\begin{equation}\label{Eq:ExampleIntegral1}
   \int_{-\infty}^{+\infty}
  \beta(\frac{X}{\varepsilon}; \omega) |D_X \xi(X)|^2 \, dX ~
\end{equation}
as well as 
\begin{equation}\label{Eq:ExampleIntegral2}
   \int_{-\infty}^{+\infty}  \bigl( \beta D_x \tanh(hD_x)\beta \bigr)
 (\frac{X}{\varepsilon}) 
    |D_X \xi(X)|^2 \, dX ~. 
\end{equation}
In our previous work \cite{CGNS05}, expressions of this form are analyzed
under the hypothesis that $\beta$ was a periodic function of $x$.
In the present paper, we are concerned with the case in which the
bottom variations $\beta(x, \omega)$ are decorrelated over large
spatial scales, which is quantified with a mixing condition on 
$\Omega$.
%%%%%%%%%%%%%%%%%%%%%%%%%%%%%%%%%%%%%%%%%%%%%%%%%
%%
%%  Subsection: Stationary ergodic processes 
%%   and mixing
%%
%%%%%%%%%%%%%%%%%%%%%%%%%%%%%%%%%%%%%%%%%%%%%%%%

\subsection{Stationary ergodic processes and mixing}
\label{Subsection:StationaryProcesses}

We take our statistical ensemble of random bottom variations
of the fluid domain to be modeled by a stationary ergodic process
which will possess some properties of mixing. Mathematically, 
given a probability space 
$(\Omega, {\mathcal M}, {\sf P})$ equipped with a group of 
${\sf P}$--measure preserving translations
$\{\tau_y: y\in{\mathbb R}\}$, and a function
$G : \Omega \to {\mathbb R}$, then a stationary process
$\gamma$ is given by  $\gamma(x;\omega) := G(\tau_x\omega)$.
The notation for the probability of a set 
$A \in {\mathcal M}$ is ${\sf P}(A)$, and integrals 
of functions $F$ over this probability space are denoted by
\begin{equation}
   \int_\Omega F \, dP = {\sf E}(F) ~.
\end{equation}
We further require that the measure be ergodic with respect 
to $\{ \tau_y \}_{y \in {\mathbb R}}$, meaning that for
any bounded measurable  function $F : \Omega \to {\mathbb R}$,
then for ${\sf P}$-almost every realization $\omega$,
\begin{equation}
   \lim_{L \to \infty} \frac{1}{L} \int_0^L F(\tau_y \omega) \, dy 
   = {\sf E}(F) ~.
\end{equation} 

 For our purposes,  we would like to take $\Omega := C({\mathbb R})$ 
the space of bounded continuous functions, for which the 
one-parameter group of translations is just that,
$(\tau_y \gamma)(\cdot) = \gamma(\cdot+y)$, for $y \in {\mathbb R}$. 
However it turns out that our sample space $C({\mathbb R})$ must be 
enlarged to a subset of the space of tempered distributions ${\mathcal S}'$,
as the process of taking limits invokes Donsker's invariance
principle, and the support of our limiting measures is on
distributions corresponding to one (or several) derivatives 
of Brownian motion. 
The modeling of a random bottom will require properties of
asymptotic independence of typical realizations with respect to
the probability measure $({\mathcal M}, {\sf P})$, specifically that
the translations $\{ \tau_y \}_{y \in {\mathbb R}}$ exhibit a
mixing property with respect to it. There are several notions of
mixing in the literature \cite{DK94}. For simplicity, we 
adopt the notion of uniform strong mixing (called $\alpha$--mixing),
although weaker conditions would also work in our setting. 
The stationary process defines a natural filtration on the probability
space given by the $\sigma$--algebras $\mathcal{M}^u_v=
\sigma(\gamma(y,\omega): v\leq y\leq u)$.
The notion of $\alpha$--mixing is that 
there is a bounded function $\alpha(y)$ 
for which $\alpha(y) \to 0$ as $y \to \infty$ 
such that for any two sets $A \in {\mathcal M}^\infty_0$ and  
$B \in {\mathcal M}^0_{-\infty}$ then
\begin{equation} \label{Eq:MixingCondition}
   |{\sf P}(A \cap \tau_y(B)) - {\sf P}(A){\sf P}(B) | 
     < \alpha(y) ~.
\end{equation}
Note that mixing implies the process is ergodic.
So that Donsker's invariance principle will extend to this 
mixing process \cite{OY72}, we require that 
$\alpha(y)=\mathcal{O}(1/y\log(y))$ for $y \mapsto +\infty$ 
as well as  
\begin{equation}\label{Eq:MixingRate}
   \int_0^\infty \alpha(y) \, dy < +\infty ~.
\end{equation}
The integral \eqref{Eq:ExampleIntegral2} 
involves a nonlocal expression in the bottom variations 
$\beta(x)$, implying that the random processes we are led 
to analyse will never be perfectly decorrelated under any
finite translation. Indeed, the spatial decay of the 
kernel of the operator $D\tanh(hD)$ implies a lower 
bound on $\alpha(y)$ of the form
$$
   \alpha(y) > e^{-2hy} ~,
$$
even for statistics of the actual realizations of the bottom
variations $\beta(x, \omega)$ which are fully decorrelated under 
sufficiently large finite translations $|y| > R$.

For the zero mean process $\gamma$, 
define the covariance function $\rho_\gamma$ to be
\begin{equation}
   \rho_\gamma(y) := {\sf E}(\gamma(0;\omega)\gamma(y;\omega))
     =  {\sf E}(\gamma(0;\omega) \tau_y\gamma(0;\omega)) ~,
\end{equation}
which is an even function of $y$ 
(\cite{Cr67} page 123,
 or \cite{B68}, page 178).
The variance $\sigma_\gamma^2 $ is given by the expression 
$$
   \sigma_\gamma^2 :=  2\int_0^\infty \rho_\gamma(y) \, dy ~.
$$
The integral exists because of the hypothesis of mixing
of the underlying process. The variance can take on any
value in $[0, +\infty)$, and we are principally concerned 
with the situation in which $\sigma_\gamma > 0$. To this
end we note the following fact.

\begin{Lemma}
When the process $\beta(x, \omega) = \partial_x \gamma(x, \omega)$,
for $\gamma(x) \in C^1$, a zero-mean, stationary process with the
above mixing properties,  then 
$$
   \sigma_\beta = 0 ~.
$$
\end{Lemma}

\smallskip\noindent
\begin{proof}
By definition,
\begin{eqnarray}
  \sigma_\beta^2 & = & 2 \int_0^{+\infty} {\sf E}(\beta(0)\beta(y))  dy  
   = 2  \int_0^{+\infty} {\sf E}(\beta(x)\beta(x+y))dy \\
   & = & 2 \int_0^{+\infty} {\sf E}(\partial_x \gamma(x)
 \partial_x\gamma(x+y)) dy \nonumber = 2 \int_0^{+\infty} {\sf E}(\partial_x \gamma(x)
   \partial_y\gamma(x+y)) dy  \nonumber \\
   & = &  2 \int_0^{+\infty}  \partial_y {\sf E}(\partial_x \gamma(x)
    \gamma(x+y)) dy ~. \nonumber 
\end{eqnarray}
Therefore by integrating, 
$$
   \sigma_\beta^2 = - 2{\sf E} (\partial_x \gamma(x) \gamma(x)) 
   +  2{\sf E} (\partial_x \gamma(x) \gamma(x + y))|_{y=+\infty} 
   = - {\sf E} (\partial_x \gamma^2(x)) ~,
$$
because the process is mixing. Using the hypothesis of ergodicity,
\begin{equation}
  {\sf E} (\partial_x \gamma^2(x)) = \lim_{T\to \infty} \frac{1}{T}
   \int_0^T \partial_x \gamma^2(x) \, dx 
   = \lim_{T\to \infty} \frac{1}{T} ( \gamma^2(x))\Big|_{x=0}^T = 0 ~.
\end{equation}
Thus the most interesting processes are those which are not derived
from derivatives of another stationary process; this fact will 
be reflected in our analysis of the asymptotics of the integrals 
\eqref{Eq:BasicMultiScaleIntegral} in the next section.  
\end{proof}

%%%%%%%%%%%%%%%%%%%%%%%%%%%%%%%%%%%%%%%%%%%%%%%%%
%%
%%   Subsection:Scale separation
%%
%%
%%%%%%%%%%%%%%%%%%%%%%%%%%%%%%%%%%%%%%%%%%%%%%%%%

\subsection{Scale separation}
\label{Section:ScaleSeparation}

The asymptotic analysis of Hamiltonians or partial differential
equations which involve random coefficients needs to establish a 
clear criterion with which to characterize terms by their order parameter.
In our present analysis, we view each term as a tempered distribution 
in space and time, namely in ${\mathcal S}'({\mathbb R}^2)$. We consider 
a term  $a(X,t; \varepsilon)$ 
to be of {\em order} $\BigOh{\varepsilon^r}$ if for any Schwartz class 
test function $\varphi(X,t)$ the limit 
$\lim_{\varepsilon \to 0} \varepsilon^{-r} 
\int a(X,t; \varepsilon) \varphi(X,t) \, dXdt$ exists. 
In this context, the terms of a partial differential
equation with random coefficients represent random ensembles 
of tempered distributions, say $\{ a(X,t; \omega, \varepsilon): 
\omega \in \Omega \} \subseteq {\mathcal S}'({\mathbb R}^2) $, which 
we state to be of order $\BigOh{\varepsilon^r}$ if for any test 
function $\varphi(X,t) \in {\mathcal S}({\mathbb R}^2)$ the 
probability measures $d{\sf P_\varepsilon}$ of 
$\varepsilon^{-r} \int a(X,t; \varepsilon, \omega) \varphi(X,t) \, dXdt$
converges weakly to some $d{\sf P_0}$. In this section we
discuss the behavior of such terms in the form 
\begin{equation}
\begin{array}{c}
   \int \gamma(\frac{X}{ \varepsilon},t; \omega) v(X, t) \varphi(X,t) \, dXdt
\\
\noalign{\vskip6pt}
  \int \gamma_1(\frac{X}{\varepsilon},t; \omega)
 \gamma_2(\frac{X+c t}{\varepsilon},t; \omega)
 v(X, t) \varphi(X,t) \, dXdt
\end{array}
\label{special_int}
\end{equation}
where $\gamma$ is a stationary mixing process, $v$ is a
solution to one of the several differential equations under
discussion, and $\varphi$ plays the r\^ole of a test function. 

\begin{Lemma}\label{Lemma3.2}
For $\gamma(x;\omega)$ a stationary ergodic process 
and for
% $f(X) \in {\mathcal S}({\mathbb R})$,
$f(X) \in L^1({\mathbb R})$,
 then for ${\sf P}$-a.e. 
realization $\omega$, 
\begin{equation}
   \int_{-\infty}^{+\infty} f(X) \gamma(\frac{X}{\varepsilon};\omega) \, dX 
   = {\sf E}(\gamma) \int_{-\infty}^{+\infty} f(X) \, dX + o(1) ~.
\label{ergodictheor}
\end{equation}
\end{Lemma}

\noindent
%{\em Proof:}
\begin{proof}
For a Schwartz class function $f$ we have
\begin{eqnarray}
%\begin{array}{lcl}
 &&  \int _{-\infty}^{+\infty}f(X)\gamma(\frac{X}{\varepsilon};\omega) dX
   =  \varepsilon \int _{-\infty}^{+\infty}f(X) \frac{d}{dX} (\int_0^{\frac{X}{
  \varepsilon}} \gamma(s;\omega) ds) dX \nonumber\\
%\noalign{\vskip7pt}
  &&\qquad =  - \int _{-\infty}^{+\infty} X f'(X) \frac{\varepsilon}{X} 
     \int_0^{\frac{X}{\varepsilon}} \gamma(s;\omega) ds dX. 
%\end{array}
\end{eqnarray}
As $\varepsilon \to 0$, combining Birkhoff ergodic theorem  
\begin{equation} 
   \frac{\varepsilon}{X} \int_0^{\frac{X}{\varepsilon}} \gamma(s;\omega) \, ds 
   \to {\sf E}(\gamma)
\end{equation} 
 with the  dominated convergence theorem
leads to
\begin{equation}
   \int _{-\infty}^{+\infty}f(X)\gamma(\frac{X}{\varepsilon};\omega) \, dX
   \to - {\sf E}(\gamma)\int _{-\infty}^{+\infty} Xf'(X) dX
\end{equation} 
and finally (\ref{ergodictheor}). 
In fact it suffices that $f \in L^1({\mathbb R})$ for the result
to hold. 
\end{proof}

The immediate application of the lemma is to the integrals 
\eqref{Eq:ExampleIntegral1}\eqref{Eq:ExampleIntegral2}, at
least to the order implied by Lemma~\ref{Lemma3.2} for their mean 
values. Under the assumption that $\xi(X) \in H^1({\mathbb R})$,
the first of these vanishes up to order $o(1)$ as
${\sf E}(\beta) = 0$, at least for ${\sf P}$-a.e. realization
$\omega$. What is clear is that the fluctuations of
\eqref{Eq:ExampleIntegral1} will play an important r\^ole 
in the derivation of the appropriate Hamiltonian equations 
of motion. The second integral \eqref{Eq:ExampleIntegral2}
   is less straightforward, as 
the mixing condition \eqref{Eq:MixingRate} is in competition 
with the integral operators represented by the Fourier multiplier
operators of the expression. We have that 
\begin{eqnarray}\label{Eq:MeanValueMultiplier}
&&   \int \bigl( \beta(x) D_x \tanh(hD_x) \beta(x) 
      \bigr)\big|_{x=\frac{X}{\varepsilon}} |D_X \xi(X)|^2 \, dX \nonumber \\
&&  \to  {\sf E}(\beta D_x \tanh(hD_x) \beta ) \int |D_X \xi(X)|^2 \, dX.
\end{eqnarray}
There are two things to discuss with this statement. The first 
is that whenever $\gamma(x, \omega) \in C^1$  
is stationary with regard to some probability space
 $(\Omega, {\mathcal M}, {\sf P})$,
then an order zero Fourier multiplier operator applied to $\gamma(x)$ 
is also stationary. Indeed, translation is respected
\begin{eqnarray}
   m(D_x)\gamma(x,\tau_y\omega) &=& \frac{1}{2\pi} \int e^{ik(x-x')} m(k) 
      \gamma(x', \tau_y\omega) \, dx'dk \\
  &=& \frac{1}{2\pi} \int e^{ik(x-x')} m(k) 
      \gamma(x'-y, \omega) \, dx'dk \nonumber \\
   & =& \frac{1}{2\pi} \int e^{ik((x-y)-x')} m(k) 
      \gamma(x', \omega) \, dx'dk \nonumber \\
  &=& m(D_x)\gamma(x-y,\omega) ~.
\end{eqnarray}
Furthermore, continuous functions of $\gamma \in C^1$, such as 
$g(\gamma) = (\gamma m(D_x)\gamma)(0)$ are measurable. By the 
ergodic theorem, for any bounded measurable $F$
\[
   \lim_{L\to \infty} \frac{1}{L}\int_0^L F(\tau_x g(\gamma)) \, dx 
    = {\sf E}(F(g)) ~,
\]
and therefore the process $\tau_x g(\gamma)$ is ergodic. 
Secondly, the expectation values of quadratic functions of $\gamma$
may be computed from the covariance function $\rho_\gamma$ 
of the stationary process. For example, 
\begin{eqnarray}\label{Eq:StationaryFourierMultipliers}
  & & {\sf E}(\gamma m(D_x)\gamma) = \lim_{y \to 0} 
     {\sf E}(\gamma(x) m(D_x)\gamma(x-y)) \nonumber \\
  & &  = \lim_{y \to 0} {\sf E}( m(-D_y)\gamma(x)\gamma(x-y))
    = \lim_{y \to 0} m(-D_y) \rho_\gamma(y)  \nonumber \\
  & & = m(-D_y) \rho_\gamma(0) ~.
\end{eqnarray}
Using these two facts, \eqref{Eq:MeanValueMultiplier} is verified
as the principal contribution from integral \eqref{Eq:ExampleIntegral2}.

\medskip

\begin{Lemma}\label{Lemma3.3} \cite{B68}
Suppose that $\beta(x; \omega)$ is a stationary ergodic process which
is mixing, with a rate $\alpha(y)$ which satisfies the condition 
\eqref{Eq:MixingRate}.
Assume that  ${\sf E}(\beta) = 0$ and that $\sigma_\beta \ne 0$.  Define
\begin{equation}\label{Eq:Y-beta}
   Y_\varepsilon(\beta)(X) 
   = \frac{\sqrt{\varepsilon}}{\sigma_\beta} \int_0^{\frac{X}{\varepsilon}} 
     \beta(y) \, dy ~.
\end{equation}
As $\varepsilon$ tends to zero, we have, in the sense 
convergence in law  that 
\begin{equation}
   Y_\varepsilon(\beta)(X) \rightharpoonup B(X) ~,
\end{equation}
where $B_\omega(X)=B(X)$ is a normalized Brownian motion.
\end{Lemma}
In particular,  
let $f(X) \in {\mathcal S}$ be a Schwartz class function, then 
\begin{eqnarray}
&&   \frac{1}{\sigma_\beta \sqrt{\varepsilon}} Z_\varepsilon(\beta, f)
   :=  \int_{-\infty}^{+\infty} \frac{1}{\sigma_\beta \sqrt{\varepsilon}}
    \beta(\frac{X}{\varepsilon}) f(X) \, dX  \label{Zepsilon}\\
 &&  = \int_{-\infty}^{+\infty} Y'_\varepsilon(\beta)(X) f(X) \, dX 
 = \int_{-\infty}^{+\infty} -\partial_X f(X) B(X) \, dX 
   + o(1) ~.  \nonumber
\end{eqnarray} 
%\end{Lemma}
This is to say that under the mild condition of mixing given in
\eqref{Eq:MixingRate}, the integrals in question converge to a
canonical stationary process, for which only two parameters are 
distinguished, the mean value ${\sf E}(\beta)$ and the 
variance $\sigma_\beta^2 $. This canonical process is given by white 
noise, 
\begin{equation}
   \int_{-\infty}^{+\infty} \beta(\frac{X}{\varepsilon}) f(X) \, dX
   =  \int_{-\infty}^{+\infty} \bigl( {\sf E}(\beta)  
   + \sqrt{\varepsilon}\sigma_\beta \partial_X B(X) \bigr)
     f(X) \, dX + o(\sqrt{\varepsilon}),
\end{equation}
where the equality is in the sense of convergence in law.
The function $f(X)$ in the integrand must be sufficiently smooth
for the latter quantities to have a mathematical sense. In fact 
we consider the operation of multiplication by $\beta(\frac{X}{\varepsilon})$
to be in the distributional sense, which has for a limit the 
distribution 
%$\sqrt{\varepsilon}\sigma_\beta \Gamma_\omega(X) := 
$\sqrt{\varepsilon}\sigma_\beta \partial_X B(X) \in {\mathcal S}'$.
This is given a precise statement in the following lemma. 

\begin{Lemma}\label{Lemma3.4}
As a distribution, multiplication by $\beta(X/\varepsilon)$
has a canonical limit in ${\mathcal S}'$. Indeed, for $f \in {\mathcal S}$, 
\begin{equation}
   \beta(\frac{X}{\varepsilon}) f(X) =  {\sf E}(\beta) f(X) 
   + \sqrt{\varepsilon} \sigma_\beta \partial_X B(X) f(X)
   + o(\sqrt{\varepsilon}) ~.
\end{equation}
\end{Lemma}

\smallskip\noindent
\begin{proof}
 Test the quantity above with a Schwartz class 
function $\varphi(X)$;
\begin{eqnarray}
   & &\int  \beta(\frac{X}{\varepsilon}) f(X) \varphi(X) \, dX
   \\
   & &\quad = {\sf E}(\beta) \int \bigl( f(X) \varphi(X) \bigr) \, dX
     - \sqrt{\varepsilon}\sigma_\beta \int B(X) 
       \partial_X (f\varphi) \, dX + o(\sqrt{\varepsilon})   \nonumber \\
   & & \quad = \int \bigl( {\sf E}(\beta)  
     +  \sqrt{\varepsilon}\sigma_\beta \partial_X B(X) \bigr)
        f(X)\varphi(X) \, dX + o(\sqrt{\varepsilon}) ~.   \nonumber
\end{eqnarray}

\noindent
This is to say that for each $f$, the random variable
$Z_\varepsilon(\beta, f)$ given in (\ref{Zepsilon})  is asymptotically normally
distributed. Given two functions $f, g \in {\mathcal S}$, the covariance
function ${\sf E}(Z_\varepsilon(\beta, f)Z_\varepsilon(\beta, g))$
can be computed in the limit as $\varepsilon \to 0$. Indeed
\begin{equation}
\begin{array}{l}
 {\sf E}(Z_\varepsilon(\beta, f)Z_\varepsilon(\beta, g))  =  
   \frac{1}{\varepsilon} \int\int \rho_\beta(\frac{X-X'}{\varepsilon})
     f(X)g(X') \, dXdX' \\
\noalign{\vskip7pt}
 \quad  = \int\int \rho_\beta(x') f(X) g(X - \varepsilon x') \, dXdx' \\
\noalign{\vskip7pt}
\quad   = \int\int \rho_\beta(x') f(X) \bigl( g(X) - \varepsilon x'
     \partial_X g(X) + \frac{\varepsilon^2}{2} x'^2 \partial_X^2 g(X) 
   + \dots \bigr)dX dx'.
\end{array}
\nonumber
\end{equation}
Noting that the term at order 
$\varepsilon$ vanishes because $\rho_\beta$ is an even function,
we have 
\begin{equation}
\begin{array}{c}
 {\sf E}(Z_\varepsilon(\beta, f)Z_\varepsilon(\beta, g))
   =  \int \rho_\beta(x') \, dx' \int f(X)g(X) \, dX  \\
\noalign{\vskip6pt}
     - \frac{\varepsilon^2}{2} \int x'^2 \rho_\beta(x') \, dx' 
     \int \partial_X f(X) \partial_X g(X) \, dX + \dots 
\end{array}
\end{equation}
In the limit as $\varepsilon \to 0$, this quantity converges to
\begin{equation}
   {\sf E}(Z_0(f) Z_0(g)) = \sigma_\beta^2 \int f(X) g(X) \, dX ~,
\end{equation}
where $Z_0(f) =\sigma_\beta \sqrt{\varepsilon}
 \int f(X) \partial_X B(X) \, dX$. This
expression is consistent with the covariance of the white noise 
process being given by $\sigma_\beta^2 \delta(X-X')$. 
\end{proof}

In the case of a process $\beta(x)$ for which 
$\sigma_\beta = 0$, the limit process for $Y_\varepsilon(X)$
is of a different character. In particular, consider a stationary 
 mixing process which is the derivative of another stationary 
process. Indeed let 
$\gamma(x) \in C^{r+1}({\mathbb R})$, and set 
$\beta(x) = \partial_x^r \gamma(x)$. 
Automatically ${\sf E}(\beta) = 0$ and $\sigma_\beta=0$. 
In this situation we have a different asymptotic result for the
behavior of integrals such as in \eqref{Eq:BasicMultiScaleIntegral}.

\begin{Lemma}\label{Lemma3.5}
Suppose that $\gamma(x) \in C^{r+1}({\mathbb R})$ is a stationary
ergodic process which satisfies the mixing condition
\eqref{Eq:MixingRate}, and set $\beta(x) = \partial_x^r \gamma(x)$.
 Then the process $\beta(X/\varepsilon)$ is asymptotic in the 
sense of distributions to higher derivatives of Brownian motion. That is, for
$\varphi(X) \in {\mathcal S}$ we have 
\begin{equation}
   \int \beta(\frac{X}{\varepsilon}) \varphi(X) \, dX = 
   \varepsilon^{r+1/2} \sigma_\gamma \int \partial_X^{r+1} B(X)
  \varphi(X)  \, dX + o(\varepsilon^{r+1/2}) ~. 
\end{equation} 
\end{Lemma}  

\smallskip\noindent
\begin{proof}
Using $\varphi(X)$ as a test function,
\begin{eqnarray}
   \int \beta(\frac{X}{\varepsilon}) \varphi(X) \, dX 
   &=& \int \partial_x^r \gamma(\frac{X}{\varepsilon}) \varphi(X) \, dx
   \nonumber\\
    & =& (-1)^r \varepsilon^r \int \gamma
     (\frac{X}{\varepsilon}) 
       \partial_X^r \varphi(X) \, dX \nonumber\\ 
   &=& (-1)^{r+1} \varepsilon^{r+ 1/2} \sigma_\gamma 
   \int Y_\varepsilon(\gamma)(X) \partial_X^{r+1} \varphi(X) \, dX
   \nonumber \\
   &=& \varepsilon^{r+ 1/2} \sigma_\gamma \int 
       \partial_X^{r+1} B(X) \varphi(X) \, dX 
    + o(\varepsilon^{r+1/2}) ~.  \nonumber   
\end{eqnarray}
\end{proof}

There are  further technical results that we will use repeatedly 
in the analysis of the equations in the KdV asymptotic regime, 
having to do with limits in the sense of tempered distributions 
of products of scaled processes. In this context, consider 
 $\gamma = (\gamma_1, \gamma_2)$ a vector  of stationary processes
which satisfy the mixing conditions 
\eqref{Eq:MixingCondition}\eqref{Eq:MixingRate}. Consider their
product $\gamma_1(X/\varepsilon)\gamma_2((X+ct)/\varepsilon)$
for some nonzero constant $c$ as a tempered distribution in the limit
 $\varepsilon \to 0$ . Define the 
{\em covariance matrix} of the vector process by
$$
   C(\gamma) = \begin{pmatrix} \sigma^2_1 & \rho_{12} \\
                                \rho_{12} & \sigma^2_2 
               \end{pmatrix}
$$
where
\begin{equation}
    \sigma^2_j = 2 \int_0^\infty {\sf E}(\gamma_j(0)\gamma_j(y)) \, dy ,
     \qquad
    \rho_{12} = \rho_{21} = \int_{-\infty}^\infty 
      {\sf E}(\gamma_1(0)\gamma_2(y)) \, dy  ~.
\end{equation}

\begin{Lemma}\label{Lemma3.6a}
If the vector process $\gamma = (\gamma_1, \gamma_2)$ is 
stationary and satisfies the mixing conditions
\eqref{Eq:MixingCondition}\eqref{Eq:MixingRate}, then the process
\begin{equation}
   Y_\varepsilon(\gamma) = \sqrt{\varepsilon} 
     \Big( \int_0^{\frac{X}{\varepsilon}} \gamma_1(y) \, dy ,
           \int_0^{\frac{X}{\varepsilon}} \gamma_2(y) \, dy \Big)
\end{equation}
converges to the two-dimensional Brownian motion 
$B(X) = (B_{1}(X), B_{2}(X))$ with 
covariance matrix $C(\gamma)$. 
\end{Lemma}

This result is analogous to  Lemma \ref{Lemma3.3} in the
vector process case. From it, we derive the next useful 
result on products of two mixing processes.

\begin{Lemma}\label{Lemma3.6}
Suppose that $(\beta_1(x), \beta_2(x))$ is a
$C^1({\mathbb R})$ vector stationary ergodic process which 
satisfies the mixing condition 
\eqref{Eq:MixingCondition}\eqref{Eq:MixingRate}, and let 
$c$ be a nonzero constant. The new process formed by the product 
$\varepsilon^{-1}\beta_1(X/\varepsilon)\beta_2((X + ct)/\varepsilon)$
converges in the sense of distributions on space-time to products 
of derivatives of a pair of Brownian motions with covariance 
matrix $C(\beta)$. More precisely, for a test function 
$\varphi(X,t) \in {\mathcal S}$ then
\begin{equation}
\begin{array}{c}
  \int \beta_1(\frac{X}{\varepsilon})\beta_2(\frac{X + ct}{\varepsilon})
      \varphi(X,t) \, dXdt  \\
\noalign{\vskip6pt}
  = \varepsilon 
      \int \partial_X B_{1}(X) \partial_X B_{2}(X+ct) 
      \varphi(X,t) \, dXdt + o(\varepsilon),  
\end{array}
  \label{Eq:Process-beta} 
\end{equation}
where the covariance matrix of $(B_{1}(X), B_{2}(X))$
is given by $C(\beta)$. 
In case $\beta_j = \partial_x^{r_j}\gamma_j$ for indices $j = 1,2$,
with $\gamma_j \in C^{r_j+1}({\mathbb R})$ (so that 
$\sigma_{\beta_j} = 0$ if $r_j \not= 0$) the new process satisfies
\begin{equation}
\begin{array}{c}
  \int \beta_1(\frac{X}{\varepsilon})\beta_2(\frac{X + ct}{\varepsilon})
    \varphi(X,t) \, dXdt  \\
\noalign{\vskip6pt}
 = \varepsilon^{r_1+r_2+1} \int \partial_X^{r_1+1} B_{1}(X) 
       \partial_X^{r_2+1} B_{2}(X+ct) \varphi(X,t)
    dXdt + o(\varepsilon^{r_1+r_2+1}),
\end{array}
\end{equation}  
where $(B_{1}(X), B_{2}(X))$ are $C(\gamma)$-correlated.
\end{Lemma}

\noindent
\begin{proof}
%{\em Proof:}
Start with the case in which both $\sigma_{\beta_j}$ are nonzero, and
write
\begin{equation}
\begin{array}{l}
  \int \beta_1(\frac{X}{\varepsilon})\beta_2(
\frac{X+ct}{\varepsilon}) \varphi(X,t)
      \, dX dt 
    = \int \beta_1(\frac{X}{\varepsilon})\beta_2(\frac{X'}{\varepsilon})
    \varphi(X, \frac{X'-X}{c}) \, \frac{dX dX'}{c} \nonumber\\
\noalign{\vskip7pt}
   = \varepsilon^2 \int \partial_X \Bigl( \int_0^{\frac{X}{\varepsilon}}
  \beta_1(\tau) \, d\tau \Bigr) \partial_{X'} \Bigl( \int_0^{
\frac{X'}{\varepsilon}}
  \beta_2(\tau') \, d\tau' \Bigr) \varphi(X, \frac{X'-X}{c}) \,
  \frac{dXdX'}{c}     \\
 \noalign{\vskip7pt}
   = \varepsilon    %% \sigma_{\beta_1} \sigma_{\beta_2} 
    \int \Bigl( \sqrt{\varepsilon} \int_0^{
\frac{X}{\varepsilon}} \beta_1(\tau) \, d\tau \Bigr) 
         \Bigl( \sqrt{\varepsilon} \int_0^{
\frac{X'}{\varepsilon}}\beta_2(\tau') \, d\tau' \Bigr) 
            \partial_X\partial_{X'}\varphi(X, \frac{X'-X}{c}) \,
    \frac{dXdX'}{c}. 
\end{array}
\end{equation} 
The latter expression is a continuous function of the processes
$Y_\varepsilon(\beta) = (Y_\varepsilon(\beta_1), Y_\varepsilon(\beta_2))$ 
of equation \eqref{Eq:Y-beta}, which itself converges in law to 
two-dimensional Brownian motion with covariance matrix $C(\beta)$ as 
described by Donsker's invariance principle. Therefore the 
asymptotic expression for \eqref{Eq:Process-beta} is given by 
\begin{equation}
   \varepsilon \int
     \partial_X B_{1}(X) \partial_X B_{2}(X+ct) \varphi(X, t) \,
     dXdt ~
\end{equation}
where $B_1(X)$ and $B_2(X)$ are two copies of Brownian
motions with the correlation matrix $C(\beta)$.
The general case reduces to the above particular case through 
integrations by parts. Indeed
\begin{eqnarray}
   \int \beta_1(\frac{X}{\varepsilon})\beta_2(
\frac{X + ct}{\varepsilon})
      \varphi(X,t) \, dXdt   \nonumber\\
  = \varepsilon^{r_1+r_2} \int \partial_X^{r_1}\gamma_1(\frac{X}{\varepsilon})
      \partial_X^{r_2} \gamma_2(
\frac{X + ct}{\varepsilon}) \varphi(X,t) \, dXdt \cr
   = (-1)^{r_1+r_2} \frac{\varepsilon^{r_1+r_2}}{ c^{r_2}} \int
      \gamma_1(\frac{X}{\varepsilon}) \gamma_2(
\frac{X+ct}{\varepsilon}) 
      \partial_X^{r_1} \partial_t^{r_2} \varphi(X,t) \, dXdt,
\end{eqnarray}
which reduces the problem to the previous case. 
\end{proof}

%% ATTENTION  $ \beta1, \beta2 $???
%%%%%%%%%%%%%%%%%%%%%%%%%%%%%%%%%%%%%%%%%%%%%%%

There is another integral that needs to be evaluated in our further
analysis. It has the form
\begin{equation}\label{test_integral}
\int_{X,t} \int_X^{X+\sqrt{gh}t} \beta\big(\frac{X}{\varepsilon})
\beta\big(\frac{\theta}{\varepsilon}) \varphi(\theta,X,t) d\theta dXdt .
\end{equation}
The next lemma shows that such integrals have  probability
measures whose weak limits converge with order at least $\BigOh{\varepsilon}$.

\begin{Lemma}\label{Lemma3.8}
Suppose that $(\beta_1(x), \beta_2(x))$ is a $C^1({\mathbb{R}})$
 vector stationary ergodic
process  which satisfy the mixing conditions
\eqref{Eq:MixingCondition} and  \eqref{Eq:MixingRate}. For test functions
$\varphi(\theta, X,t) \in {\mathcal S}$,

\begin{equation}
\int dX dt \int_X^{X+\sqrt{gh}t} \left[
\beta_1\big(\frac{X}{\varepsilon})\beta_2\big(\frac{\theta}{\varepsilon})
+
\beta_2\big(\frac{X}{\varepsilon})\beta_1\big(\frac{\theta}{\varepsilon})
 \right]
\varphi(\theta,X,t) d\theta
= \BigOh{\varepsilon}.\label{test_integral2}
\end{equation}
\end{Lemma}

\begin{proof}
The integral is written as the sum of two terms, each one of the form
\begin{equation}
\begin{array}{l}
\int dXdt \int_X^{X+\sqrt{gh}t} \beta_i\big(\frac{X}{\varepsilon})
\beta_j\big(\frac{\theta}{\varepsilon}) \varphi(\theta,X,t) d\theta
\\
\noalign{\vskip6pt}
= \varepsilon \int dXdt \int_X^{X+\sqrt{gh}t}
\partial_X \Big( \sqrt{\varepsilon} \int_0^{\frac{X}{\varepsilon}} \beta_i(s) ds\Big)
\partial_\theta \Big( \sqrt{\varepsilon} \int_0^{\frac{\theta}{\varepsilon}} \beta_j(s)
ds\Big)
\varphi(\theta,X,t) d\theta
 \\
\noalign{\vskip6pt}
 = \varepsilon \int dXdt \int_X^{X+\sqrt{gh}t}
\Big( \sqrt{\varepsilon} \int_0^{\frac{X}{\varepsilon}} \beta_i(s) ds\Big)
\Big( \sqrt{\varepsilon} \int_0^{\frac{\theta}{\varepsilon}}\beta_j(s) ds\Big)
\partial_{X\theta} \varphi(\theta,X,t) d\theta
\\
\noalign{\vskip6pt}
\qquad + \varepsilon \int dXdt
\Big( \sqrt{\varepsilon} \int_0^{\frac{X}{\varepsilon}} \beta_i(s) ds\Big)
 \Big[\Big(\sqrt{\varepsilon} \int_0^{\frac{
X+\sqrt{gh}t}{\varepsilon}} \beta_j(s)
ds\Big)
\partial_{\theta}\varphi(X+\sqrt{gh}t,X,t)
 \\
\noalign{\vskip6pt}
\qquad\qquad\qquad\qquad\qquad
- \Big( \sqrt{\varepsilon} \int_0^{\frac{X}{\varepsilon}} \beta_j(s) ds\Big)
\partial_{\theta}\varphi(X,X,t) \Big] \
 \\
\noalign{\vskip6pt}
\qquad - \varepsilon \int dXdt
\Big( \sqrt{\varepsilon} \int_0^{\frac{X}{\varepsilon}} \beta_i(s) ds\Big)
 \Big[\Big( \sqrt{\varepsilon} \int_0^{\frac{X+\sqrt{gh}t}{\varepsilon}}
\beta_j(s) ds\Big)
\partial_{X}\varphi(X+\sqrt{gh}t,X,t)
\\
\noalign{\vskip6pt}
\qquad\qquad\qquad\qquad\qquad
- \Big( \sqrt{\varepsilon} \int_0^{\frac{X}{\varepsilon}} \beta_j(s) ds\Big)
\partial_{X}\varphi(X,X,t) \Big]
\\
\noalign{\vskip6pt}
\qquad -\varepsilon \int dXdt
\Big( \sqrt{\varepsilon} \int_0^{\frac{X}{\varepsilon}} \beta_i(s) ds\Big)
\Big[ \frac{1}{\sqrt{\varepsilon}} \beta_j(\frac{X+\sqrt{gh}t}{\varepsilon})
\varphi(X+\sqrt{gh}t,X,t)
\\
\noalign{\vskip6pt}
\qquad\qquad\qquad\qquad\qquad\qquad\qquad
- \frac{1}{\sqrt{\varepsilon}} \beta_j(\frac{X}{\varepsilon} )
\varphi(X,X,t) \Big]
 \\
\noalign{\vskip6pt}
\end{array}
\end{equation}
with $i,\,j\in\{1,2\}$ and $i\neq j$.
All of the terms have distributional limits which are at least
$\BigOh{\varepsilon}$. Simple cases which illustrate the estimate are:

\begin{equation}
\begin{array}{lll}
I:&=&
\varepsilon \int dXdt\left[
\Big( \sqrt{\varepsilon} \int_0^{\frac{X}{\varepsilon}} \beta_1(s) ds\Big)
\frac{1}{\sqrt{\varepsilon}} \beta_2(\frac{X}{\varepsilon} )+
\Big( \sqrt{\varepsilon} \int_0^{\frac{X}{\varepsilon}} \beta_2(s) ds\Big)
\frac{1}{\sqrt{\varepsilon}} \beta_1(\frac{X}{\varepsilon})\right]
\varphi(X,X,t)\\
\noalign{\vskip6pt}
&=& \frac{\varepsilon}{2}  \int dXdt\partial_X
\Big( \sqrt{\varepsilon} \int_0^{\frac{X}{\varepsilon}} \beta_1(s) ds
  \sqrt{\varepsilon}   \int_0^{\frac{X}{\varepsilon}} \beta_2(s) ds
\Big)
\varphi(X,X,t)
\\
\noalign{\vskip6pt}
& =&-\frac{\varepsilon}{2} \int dXdt
\Big( \sqrt{\varepsilon} \int_0^{\frac{X}{\varepsilon}} \beta_1(s) ds
 \sqrt{\varepsilon}\int_0^{\frac{X}{\varepsilon}} \beta_2(s) ds
\Big)
\partial_X \varphi(X,X,t).
\end{array}
\end{equation}
\begin{equation}
\begin{array}{lll}
II:&=&
\varepsilon \int dXdt
\Big( \sqrt{\varepsilon} \int_0^{\frac{X}{\varepsilon}} \beta_i(s) ds\Big)
\frac{1}{\sqrt{\varepsilon}} \beta_j(\frac{X+\sqrt{gh}t}{\varepsilon})
\varphi(X+\sqrt{gh}t,X,t)
 \\
\noalign{\vskip6pt}
& =&-\varepsilon\int dXdt
\Big( \sqrt{\varepsilon} \int_0^{\frac{X}{\varepsilon}} \beta_i(s) ds\Big)
\frac{1}{\sqrt{gh}} \partial_t \
\Big( \sqrt{\varepsilon} \int_0^{
\frac{X+\sqrt{gh}t}{\varepsilon}} \beta_j(s) ds\Big)
\varphi(X+\sqrt{gh}t,X,t)
\\
\noalign{\vskip6pt}
&=&- \frac{\varepsilon}{\sqrt{gh}} \int dXdt
\Big( \sqrt{\varepsilon} \int_0^{\frac{X}{\varepsilon}} \beta_i(s) ds\Big)
\Big( \sqrt{\varepsilon} \int_0^{\frac{X+\sqrt{gh}t}{\varepsilon}}
 \beta_j(s) ds\Big) \partial_t \varphi(X+\sqrt{gh}t,X,t).
\end{array}
\nonumber
\end{equation}
In these expressions, notice that the factors that appear are
continuous functionals on path space. 
Therefore, as $\varepsilon \mapsto 0$, they converge in law to 
functionals of Brownian motions \cite{B68}.  
The other remaining terms are easy to estimate.
\end{proof}

%%%%%%%%%%%%%%%%%%%%%%%%%%%%%%%%%%%%%%%%%%%%%%%%%
%%
%% Subsection: Random characteristic coordinates
%%
%%
%%%%%%%%%%%%%%%%%%%%%%%%%%%%%%%%%%%%%%%%%%%%%%%%

\subsection{Random characteristic coordinates}
\label{Subsection:RandomCharacteristicCoordinates}

Our method to derive the long-wave limit gives rise to a version
of the KdV equation which has coefficients which are realization
dependent. That is, the approximation process does not fully
homogenize, and there are persistent, realization dependent effects
that are as important as the classical effects of dispersion and of
nonlinear interactions. The principal manifestation of this is the
random overall wavespeed, expressed in the limit as 
$\varepsilon \to 0$ as 
\begin{equation}\label{Eq:Characteristics-0}
   c_0(X,\omega)
% := \sqrt{gh_\varepsilon(X)}\sim 
 =   \sqrt{gh}\Bigl( 1 
      - \frac{\varepsilon^{3/2}\sigma_\beta}{2h} 
        \partial_X B(X) -  \varepsilon^2 a_{KdV} \Bigr) ~.
\end{equation}
The constant $a_{KdV}$ is an adjustment to the characteristic velocity 
that is to be determined by an asymptotic analysis.
The normally expected procedure is to solve the characteristic
equations with this given wavespeed;
\begin{equation}\label{Eq:CharacteristicVectorField}
   \frac{dX}{dt} = c_0(X,\omega) ~, \qquad X(0) = Y ~,
\end{equation}
to obtain characteristic coordinates $(Y,t)$ describing a net
translational motion about which the more subtle nonlinear dispersive
evolution takes place. In the context of a random bottom environment, 
however, the characteristic velocity field $c_0(X,\omega)$ in 
\eqref{Eq:Characteristics-0} has a component which is white noise, 
and when the flow of the characteristic vector field 
\eqref{Eq:CharacteristicVectorField} is required, 
\eqref{Eq:Characteristics-0} is too singular to be able 
to make sense of a  solution. 

Our derivation of the KdV equation is nonetheless 
performed in characteristic coordinates. To do this, 
our alternative strategy is to use a natural regularization of the 
characteristic wavespeed given in \eqref{Eq:Characteristics-0} as
an approximation, and to consider the characteristic coordinates 
indicated by \eqref{Eq:CharacteristicVectorField} to be the limit 
as $\varepsilon \to 0$ of a sequence of more regular flows.
The regularized characteristic vector field that we use is 
\begin{eqnarray}\label{reguflow}
  && \frac{dX}{dt}  =  c_\varepsilon(X, \omega) :=
%    \sqrt{g\bigl( h - \varepsilon\beta(\frac{X}{\varepsilon}, \omega)
%%       - \varepsilon^2 a\bigr) }   \\
%    & & \sim
  \sqrt{gh}\Bigl( 1 - \frac{\varepsilon}{2h} \beta(\frac{X}{\varepsilon}) 
      -  \varepsilon^2 a_{KdV} \Bigr)
%      - \frac{\varepsilon^2}{8h^2} \beta^2(\frac{X}{\varepsilon}) \Bigr) ~,
    \nonumber \\
   && X(0)  =  Y. \nonumber 
\end{eqnarray}
We remark that as long as $\beta(x, \omega) \in C^1({\mathbb R})$ for 
${\sf P}$-a.e. realization $\omega$, the characteristic vector field 
$c_\varepsilon(X,\omega)$ is $C^1$, and for a given realization
$\omega$ it is uniformly so in $\varepsilon$. Therefore the flows 
$X(t) = \Phi_t^\varepsilon(Y, \omega)$ exist for all $\varepsilon$, 
and lie in a bounded subset of $C^1$. The characteristics $X(t)$ are 
themselves $C^1$, and they are ordered by their initial values; if 
$Y_1 < Y_2$ then for all $t$, $X_1 = \Phi_t^\varepsilon(Y_1, \omega) 
< X_2 = \Phi_t^\varepsilon(Y_2, \omega)$. As $\varepsilon \to 0$ 
there will normally not be a $C^1$ limit of the flows, but by 
standard compactness arguments there are limits $\Phi_t^0(Y, \omega)$ 
in any $C^\alpha({\mathbb R})$, $ 0 \leq \alpha < 1$ which converge 
uniformly on compact sets, and which preserve the 
ordered property of the characteristics. Each such limit 
$X = \Phi_t^0(Y, \omega)$ can be taken to be a well-defined 
continuous and continuously invertible transformation.

To understand the asymptotic behavior of the  transformation to characteristic 
coordinates, write
\begin{equation}
   \Phi_t^\varepsilon(Y, \omega) = X^0(t) 
   + \varepsilon X^1(t) + \varepsilon^2 X^2(t) + \dots,
\end{equation}
where $X^0(0) = Y$, and $X^j(0) = 0$ for $j \geq 1$ provide the
initial conditions for the flow. Substituting this into the 
characteristic equation gives the result that
\begin{eqnarray}
   \frac{dX^0(t)}{dt} & = & \sqrt{gh} ~, \qquad 
      X^0(t) = Y + t\sqrt{gh}, \\
   \frac{dX^1(t)}{dt} & = & -\frac{1}{2}\sqrt{\frac{g}{h}} 
      \beta(X^0(t) / \varepsilon, \omega)  \nonumber \\
     &=& -\frac{1}{2}\sqrt{\frac{g}{h}} 
      \beta((Y + t\sqrt{gh}) / \varepsilon, \omega),
\end{eqnarray} 
thus, variations to the characteristics are given by
\begin{equation}
   X^1(t) =  -\frac{\varepsilon}{2h}
     \int_{Y/\varepsilon}^{(Y + t\sqrt{gh})/\varepsilon} 
     \beta(s, \omega) \, ds ~.
\end{equation}
The final term relevant to our considerations is 
\begin{equation}
%   \frac{dX^2(t)}{dt} = -\Bigl( \frac{a}{2}\sqrt{\frac{g}{h}} 
%     + \frac{1}{8}\sqrt{\frac{g}{h^3}} 
%   \beta^2(\frac{X(t)}{\varepsilon}) \Bigr),
\frac{dX^2(t)}{dt} = - \sqrt{gh}  ~a_{KdV}
\end{equation}
which integrates simply to 
$X^2(t) = - \sqrt{gh}a_{KdV} ~t.$
%%$X^2(t) = -\sqrt{g/h}a /2 + {\sf E}(\beta^2)/(8h))t$.  
Studying the integral expressions for $X^1(t)$ more closely, 
we find that  
\begin{equation}
    X^1(t) =  - \frac{\varepsilon}{2h} \Bigl( 
      \int_0^{(Y + t\sqrt{gh})/\varepsilon} \beta(s) \, ds 
      - \int_0^{Y/\varepsilon} \beta(s) \, ds \Bigr) ,
\end{equation}
which converges in law to Brownian motion as $\varepsilon \to 0$, 
according to our discussion in Section~\ref{Section:ScaleSeparation}. 
Hence
\begin{equation}
    X^1(t) = 
    - \frac{\sqrt{\varepsilon}\sigma_\beta}{2h}
    \bigl( B(Y+t\sqrt{gh}) - B(Y)  \bigr) 
    + o(\sqrt{\varepsilon}) ~.
\end{equation}
In particular, the term $\varepsilon X^1(t)$ contributes at order
$\varepsilon^{3/2}$. Due to Brownian scaling and
to the property of independence of increments, 
\begin{equation}
   X^1(t) = 
  - \frac{\sqrt{\varepsilon}\sigma_\beta}{2h}
   B_{\omega(Y)}(t\sqrt{gh}) 
   = - \sqrt{\varepsilon} \Bigl( \frac{\sigma_\beta}{2h} \sqrt[4]{gh}
        \Bigr) B_{\omega(Y)}(t) ~.
\end{equation} 
We note that the realizations $\omega(Y)$ of Brownian motion depend 
on the different initial positions $Y$, and in particular that for
distinct initial points $Y_1$ and $Y_2$ the selection of realizations 
$B_{\omega(Y_1)}(t)$ and  $B_{\omega(Y_2)}(t)$
of Brownian motion are independent, as long as $Y_2 - Y_1 > \sqrt{gh}t$.  
Putting this information together, an expression for the
characteristic flow is given by 
\begin{eqnarray} 
  &&  X(t,Y; \varepsilon, \omega)  = Y + t\sqrt{gh} 
    - \frac{\varepsilon^2}{2h}
      \int_{Y/\varepsilon}^{(Y + t\sqrt{gh})/\varepsilon} 
      \beta(s, \omega) \, ds    \cr
&& - \sqrt{gh} ~a_{KdV}
%  && - \sqrt{\frac{g}{h}}(a/2 +\frac{1}{8h} {\sf E}(\beta^2)) 
      \varepsilon^2 t + \cdots ~.
\end{eqnarray}
As $\varepsilon$ tends to $0$, the characteristics tend to the 
limiting distribution of paths given by 
\begin{equation} 
  X(t,Y; \omega) = Y + t\sqrt{gh} 
    - \frac{\varepsilon^{3/2}\sigma_\beta}{2h}
      \sqrt[4]{gh}B_{\omega(Y)}(t)   
%%      \bigl( B(Y+t\sqrt{gh}) - B(Y)  \bigr) \cr
 - \varepsilon^2 \sqrt{gh}a_{KdV} t
% && - \varepsilon^2 \sqrt{\frac{g}{h}}(\frac{a}{2} 
%%+\frac{1}{8h} {\sf E}(\beta^2)) t 
     + \cdots ~. 
\end{equation} 
Inverting the expression gives a formula for $Y$ in terms of $X$ and $t$; 
\begin{eqnarray} 
  &&  Y(t,X; \varepsilon, \omega) 
    = \Phi_{-t}^\varepsilon(X; \omega) = X - t\sqrt{gh} 
    + \frac{\varepsilon^2}{2h} 
      \int_{(X - t\sqrt{gh})/\varepsilon}^{X /\varepsilon}
      \beta(s, \omega) \, ds       \cr
  && + \varepsilon^2 \sqrt{gh}
a_{KdV}  t
%(a/2 + \frac{1}{8h} {\sf E}(\beta^2)) t
     + \cdots ~.
\end{eqnarray}
As $\varepsilon$ tends to $0$,
\begin{eqnarray} 
  && Y(t, X; \omega) = X - t\sqrt{gh} +
 \frac{\varepsilon^{3/2}\sigma_\beta}{2h}
      \sqrt[4]{gh} B_{\omega(X)}(t) 
%%     \big( B(X) - B(X-t\sqrt{gh})  \big) 
%   + \sqrt{\frac{g}{h}}(a/2 +\frac{1}{8h} {\sf E}(\beta^2))
%     \varepsilon^2 t + \cdots ~.    \nonumber 
+ \sqrt{gh} ~a_{KdV} \varepsilon^2 t + \cdots ~.    \nonumber 
\end{eqnarray}
%\begin{Lemma} \label{lemma-Jacobian}
The Jacobian of the flow has the following asymptotic expansion
\begin{equation}\label{Eq:JacobianAsymptotics}
   \frac{dX}{dY} = 1 - \frac{\varepsilon}{2h}
    \Bigl[\beta( \frac{Y +\sqrt{gh}t }{\varepsilon}) - \beta(\frac{Y}{\varepsilon})\Bigr]
   = 1 + \BigOh{\varepsilon} ~.
\end{equation}
In the limit as $\varepsilon$ tends to zero, the Jacobian 
\eqref{Eq:JacobianAsymptotics}, when multiplying a test function, 
behaves asymptotically as
\begin{equation}
  \frac{dX}{dY} \sim 1 - \frac{\varepsilon^{3/2}\sigma_\beta}{2h}
      \sqrt[4]{gh} \partial_X B_{\omega(Y)}(t) .    
\end{equation}

%%%%%%%%%%%%%%%%%%%%%%%%%%%%%%%%%%%%%%%%%%%%%%%%%
%%
%%   Section: The Boussinesq regime
%%
%%
%%
%%%%%%%%%%%%%%%%%%%%%%%%%%%%%%%%%%%%%%%%%%%%%%%%%

\section{Boussinesq regime}
\setcounter{equation}{0}

We now return to the expression (\ref{Eq:scaledhamil}) for the scaled 
Hamiltonian, in order to give a formal derivation of the appropriate 
Boussinesq system in this regime.
Recall that $\beta$ is a mean zero, stationary mixing process
 with correlation function $\rho_\beta$.
Using the analysis of the previous section, we write the leading order
contributions of the second and fourth terms  of (\ref{Eq:scaledhamil})
 in the form
\begin{equation}\label{Eq:Integral1}
   \int_{-\infty}^{+\infty}  \beta(\frac{X}{\varepsilon}, \omega)
 |D_X \xi(X)|^2 \, 
   dX ~ = \sqrt{\varepsilon}\sigma_\beta 
\int \partial_X B(X) |D_X \xi(X)|^2
   + o(\sqrt{\varepsilon})
\end{equation}
and 
\begin{equation}\label{Eq:Integral2}
   \int_{-\infty}^{+\infty}  
   \bigl( \beta D_x \tanh(hD_x)\beta \bigr) (\frac{X}{\varepsilon}) |D_X
   \xi(X)|^2 \, dX ~ = {\sf E}(\beta D_x \tanh(hD_x) \beta ) \int
 |D_X\xi(X)|^2 \, dX ~,
\end{equation}
where as in \eqref{Eq:StationaryFourierMultipliers} we calculate that 
\begin{equation}
  {\sf E}(\beta D_x \tanh(hD_x) \beta ) = (D_y  \tanh(hD_y)
  \rho_\beta)(0) := a_\beta.
\end{equation} 
The constant $a_\beta$ will contribute to  an adjustment
 of the linear wavespeed in the Boussinesq  regime.
The equalities in (\ref{Eq:Integral1})(\ref{Eq:Integral2})  
are to be taken in the sense of  law of the
corresponding random processes.

We now include these expansions into the Hamiltonian, retaining terms up to order 
${\mathcal O}(\varepsilon^5)$ and dropping those of higher order
\begin{eqnarray}
   &&H= \frac{\varepsilon^3}{2} \int \Big( (h - \varepsilon^{3/2} 
   \sigma_\beta \partial_X B(X) -\varepsilon^2 a_\beta)
 |D_X\xi|^2 + g \eta^2 \Big)dX
 \cr
&& \qquad +\frac{\varepsilon^5 }{2} 
\int  (\xi D_X \eta D_X \xi- \frac{h^3}{3} \xi D_X^4 \xi)  dX +
 o(\varepsilon^5).
\end{eqnarray} 
%%%%%We write $\Gamma_\omega(X) = \partial_X B(X)$ the process known as
%white noise and 
We note the r\^ole of a stochastic {\sl effective depth}
played by 
\begin{equation} 
  h_0(X) = h - \varepsilon^{3/2} \sigma_\beta
   \partial_X B(X) -\varepsilon^2 a_\beta +o(\varepsilon^2)
\end{equation}
which is a function of the long length scale variables alone.
Since it is normally not necessary to introduce characteristic 
coordinates in this derivation, a regularization such as described in 
section~\ref{Subsection:RandomCharacteristicCoordinates} is not required, 
and the limiting effective depth $h_0(X)$ is used directly in the averaged 
Hamiltonian. 

Changing the variables  $(\eta,\xi)$ to $(\eta, u = \partial_{X}
\xi)$, 
the Hamiltonian becomes
\begin{equation}\label{Hamil-eta-u}
   H_1= 
     \frac{\varepsilon^3}{2} \int\Big( h_0(X) u^2 + g \eta^2 
   - \varepsilon^2 ( \frac{ h^3}{3} (\partial_{X} u)^2 -
     \eta u^2 ) \Big)dX.
\end{equation}  
The symplectic structure  has to be modified  accordingly as in \cite{CGK05}.
Consider the transformation
$w \to v= f(w)$, which transforms Hamilton's equations
\begin{equation}
   \partial_t w
   = J \delta_w H(w)
\end{equation}
to the form 
\begin{equation}
   \partial_t  v
   = J_1\delta_v H_1(v)
\end{equation}
with a new symplectic structure 
\begin{equation}
   J_1= \partial_w f J (\partial_w f)^{\top} ~,
\end{equation}
where $\partial_w f$ is the Jacobian of the map $f$.
In our case,  $ w = \begin{pmatrix} \eta \\ \xi 
               \end{pmatrix} $,  $ ~v =  \begin{pmatrix} \eta \\ u
               \end{pmatrix} = \begin{pmatrix} I & 0 \\ 0 & \partial_X
\end{pmatrix} w $,
and the matrix 
$J = \varepsilon^{-3}  
     \begin{pmatrix} 0 & I \\ -I & 0
     \end{pmatrix} $
is transformed to $
   J_1 = \varepsilon^{-3}\begin{pmatrix} 0 & -\partial_X \\ -\partial_X & 0
         \end{pmatrix} $, where the power of $\varepsilon$ is due to
         the scaling transformations \eqref{BouScal}. 
The evolution equations take the form
\begin{equation}\label{Eq:HamiltonsEquations1}
   \partial_t  \begin{pmatrix} \eta \\ u 
               \end{pmatrix} 
   = J_1 \,
     \begin{pmatrix} \delta_\eta H_1  \\ \delta_u H_1 
     \end{pmatrix}.
\end{equation}
In the end we find the Boussinesq system in the form
\begin{eqnarray}\label{Boussinesqsystem}
   \partial_t \eta &=& -\partial_{X} ((h_0(X) +\varepsilon ^2 \eta) u) -
     \varepsilon^2 \frac{h^3}{3} \partial_{X}^3 u ~, \nonumber \\
   \partial_t u &=& -g \partial_{X} \eta -\varepsilon^2 u\partial_{X} u ~.
\end{eqnarray}
While this form of Boussinesq system appears most naturally from 
a direct expansion of the Hamiltonian of the problem of water waves, 
the resulting system of partial differential equations is not well
posed, and it is rarely used directly in modeling. In the present
setting, the situation is further aggravated by the fact that a 
coefficient in the above system is singular, as it involves the 
second derivative of a Brownian motion. Several routes to resolving 
these issues are possible, modifying the linear dispersion relation
for the Boussinesq system, and regularizing the coefficients as in
section~\ref{Subsection:RandomCharacteristicCoordinates}, for
example. However we will not pursue this direction of inquiry 
in the present paper, preferring to make a more systematic study 
of the KdV scaling regime.

%%%%%%%%%%%%%%%%%%%%%%%%%%%%%%%%%%%%%%%%%%%
%%
%%  Section: the KdV regime
%%
%%
%%%%%%%%%%%%%%%%%%%%%%%%%%%%%%%%%%%%%%%%%%%

\section{The KdV regime}\label{kdvregime}
\setcounter{equation}{0}
In the case of the Boussinesq derivation, the limit of certain integrals
in the water waves Hamiltonian will give rise to singular coefficients
in the resulting equations of motion.This is even more true in the case of 
the KdV regime; indeed the transformation to characteristic coordinates
will give rise to a modified symplectic structure which 
involves a second derivative of Brownian motion, something that is not
acceptable on an analytic level. To get around this difficulty, 
we regularize the linear wavespeed as described in 
section~\ref{Subsection:RandomCharacteristicCoordinates}, 
a process which consists of retaining certain terms with rapidly 
varying coefficients in the Hamiltonian, and 
only taking the limit after the long wave equations are derived.
%% The important point is that the integrals 
%$\int f(X) \beta(X/\varepsilon;\omega) dX$ 
%have different asymptotic behavior for different statistical 
%ensembles (see Lemma \ref{Lemma3.5} of Section
%\ref{Section:ScaleSeparation}).
 We assume that
$\sigma_\beta > 0$, which implies that the resulting realization
dependent fluctuations  are maximally significant in the limit, 
and we will perform the smoothing procedure in a way which is 
consistent with this assumption. 

%%%%%%%%%%%%%%%%%%%%%%%%%%%%%
%
%Subsection: Successive changes of variables
%
%%%%%%%%%%%%%%%%%%%%%%%%%
\subsection{Successive changes of variables}

We start again from the expression (\ref{Eq:scaledhamil}) for the Hamiltonian.
As in the derivation of the Boussinesq system, we first 
change the variables $(\eta,\xi)$ to $(\eta, u=\partial_X\xi)$, 
leading to a transformed  Hamiltonian $H_1^{\varepsilon}$ defined by
\begin{equation}\label{Hamil1}
   H_1^{\varepsilon} 
   = \frac{\varepsilon^3}{2} \int \Big(h_\varepsilon(X) u^2 + g \eta^2 
     - \varepsilon^2( \frac{ h^3}{3} (\partial_{X} u)^2 
     - \eta u^2) \Big) dX
\end{equation} 
and a  modified  symplectic structure 
$J_1 = \varepsilon^{-3} \begin{pmatrix} 0 & -\partial_X
 \\ -\partial_X & 0
     \end{pmatrix}$. 
%The adjustment of the effective depth $a_{KdV}$, which can be viewed
% as a perturbation of the linear wavespeed is somehow subtle
% and will be computed in the subsequent analysis.
The next change of variables is defined by the transformation
\begin{equation} \label{variables-kdv}
      \eta = \sqrt[4]{\frac{h_\varepsilon}{4g}} (r + s) ~,
   \quad u = \sqrt[4]{\frac{g}{4 h_\varepsilon}} (r - s) ~.
\end{equation}
The new symplectic structure resulting from this transformation is
\begin{equation}
   J_2 = \varepsilon^{-3}\begin{pmatrix} -\partial_X & \displaystyle{\frac{1}{4} 
     \frac{\partial_X h_\varepsilon }{h_\varepsilon}} \\
   - \displaystyle{\frac{1}{4}\frac{\partial_X h_\varepsilon }
     {h_\varepsilon}} & \partial_X
     \end{pmatrix} ~,
\end{equation}
whose off-diagonal terms quantify the scattering of solutions due to
variations in the topography. 
In this expression,  we retain the regularized expression :
\begin{equation}\label{effective-h-eps}
   h_\varepsilon(X)= h - \varepsilon \beta(\frac{X}{\varepsilon})
 -\varepsilon^2 a_{\beta}.
 % \sim \sqrt{gh}(1- 
%\frac{\varepsilon \beta(\frac{X}{\varepsilon})}{2h} +\frac{\varepsilon^2 a_\beta}
%{2h} + \frac{\varepsilon^2 \beta^2(\frac{X}{\varepsilon}) }{4h^2} )~.
\end{equation}
The Hamiltonian is written as 
\begin{eqnarray}
  &&
  H_2^{\varepsilon}(r,s) 
   = \frac{\varepsilon^3}{2} \int \Big(\sqrt{g h_\varepsilon(X)} (r^2+s^2 )
   - \frac{\varepsilon^2}{3} h^3 \Big[ \Bigl(\partial_X
     \sqrt[4]{\frac{g}{4h_\varepsilon}} r \Bigr)^2 \nonumber \\
  & & \quad\,\, - 
    2 \Bigl(\partial_X \sqrt[4]{\frac{g}{4h_\varepsilon}} r \Bigr) 
      \Bigl(\partial_X \sqrt[4]{\frac{g}{4h_\varepsilon}} s \Bigr) +
      \Bigl(\partial_X \sqrt[4]{\frac{g}{4h_\varepsilon}} s \Bigr)^2 
      \Big]
        \nonumber \\
    & & \quad\, + \, \frac{\varepsilon^2}{2} 
        \sqrt[4]{\frac{g}{4h_\varepsilon}} (r^3 - r^2 s - r s^2 + s^3) \,\Big)
        dX + \LittleOh{\varepsilon^5} ~.
\end{eqnarray}
Notice that, except for the first term in the Hamiltonian, $h_\varepsilon$ 
appears in terms that are already of order $\varepsilon^5$ and thus can 
be replaced there by the constant $h$ in this asymptotic calculation. 
Denoting by 
$$
   c_1 = \frac{h^3}{3} \sqrt{\frac{g}{4h}}, \, \, \quad
   c_2 = \frac{1}{2}\sqrt[4]{\frac{g}{4h}},
$$
we rewrite $H_2^{\varepsilon}$ in the form:
\begin{eqnarray}
   && H_2^{\varepsilon}
   =  \frac{\varepsilon^3}{2} \int \sqrt{gh_\varepsilon} (r^2+s^2 )
    - c_1 \varepsilon^2  \Big((\partial_{X} r )^2 -2 (\partial_{X} r )
      (\partial_{X} s )+ (\partial_{X} s)^2\Big)\nonumber \\
  & & \quad\,+ c_2\varepsilon^2  \Big(r^3 - r^2 s - r s^2 + s^3\Big)  \, dX
      + \LittleOh{\varepsilon^5} ~.
\end{eqnarray}
Hamilton's equations for $(r,s)$ take the form 
\begin{equation}\label{Eq:r-s}
   \partial_t  \begin{pmatrix} r \\ s 
               \end{pmatrix} 
   = J_2 \, \begin{pmatrix} \delta_r H_2^{\varepsilon} \\ 
                            \delta_s H_2^{\varepsilon} 
               \end{pmatrix} 
\end{equation}
where $\delta_r H_2^{\varepsilon}$ and $\delta_s H_2^{\varepsilon} $ 
are computed as follows:
\begin{equation}
\begin{array}{lcl}
  \delta_r H_2^{\varepsilon}& =&\frac{\varepsilon^3}{2} 
    \Big( \sqrt{g h_\varepsilon}  ~2 r 
  + c_1 \varepsilon^2 (2 \partial_{X}^2 r - 2 \partial_{X}^2 s) 
  + c_2 \varepsilon^2 (3 r^2 -2 r s -s^2) \Big)\\
\noalign{\vskip6pt}
\delta_s H_2^{\varepsilon} &=& \frac{\varepsilon^3}{2} 
    \Big( \sqrt{g h_\varepsilon} ~2 s
  - c_1\varepsilon^2 (2 \partial_{X}^2 r - 2 \partial_{X}^2 s)
  - c_2\varepsilon^2 ( r^2 +2 rs -3 s^2)\Big).
\end{array}
\end{equation}
Hamilton's equations are explicitly
\begin{equation}
\begin{array}{lcl}
\partial_t r &=& -\partial_X\Big[\sqrt{g h_\varepsilon} r + 
\varepsilon^2 \Big( c_1( \partial_X^2 r- \partial_{X}^2s) +\frac{1}{2} c_2
(3r^2 -2rs -s^2) \Big) \Big]  \\
\noalign{\vskip6pt}
&\quad&  +\frac{1}{4}  \frac{ \partial_X h_\varepsilon}{h_\varepsilon} 
\Big[ \sqrt{g h_\varepsilon}s +\varepsilon^2 \Big( c_1( \partial_{X}^2 s
- \partial_{X}^2 r) +\frac{1}{2} c_2 (-r^2-2rs+3s^2) \Big) \Big]
\end{array}
\label{eqforr} 
\end{equation}
\begin{equation}
\begin{array}{lcl}
\partial_t s &=& \partial_X\Big[\sqrt{g h_\varepsilon}s + 
\varepsilon^2 \Big( c_1( \partial_{X}^2 s- \partial_{X}^2 r)+\frac{1}{2} c_2
(-r^2-2rs+3s^2) \Big) \Big]  \\
\noalign{\vskip6pt}
&\quad &-\frac{1}{4}  \frac{ \partial_X h_\varepsilon}{h_\varepsilon} 
\Big[ \sqrt{g h_\varepsilon}r+\varepsilon^2 \Big( c_1( \partial_{X}^2 r
- \partial_{X}^2 s)+\frac{1}{2} c_2 (3r^2-2rs-s^2) \Big) \Big].
\end{array}
\label{eqfors}
\end{equation}
In the action of $J_2\delta H_2^{\varepsilon}$, there 
are products of $h_\varepsilon$ and its derivatives, and each factor
tends to a distribution (see Lemma \ref{Lemma3.2}) in the limit
$\varepsilon \to 0$. The product is nevertheless well 
defined because of the form it takes:
$$
   \int  \partial_X h_\varepsilon ~ {h_\varepsilon}^{-1/2}
   f(X) \, dX = 2 \int \partial_X  h_\varepsilon ^{1/2} f(X) \, dX ~.
$$
We perform an additional change of scale of $s$ relative to $r$ defined by 
\begin{equation}\label{Eq:InhomogeneousScaling}
               \begin{pmatrix} r \\ s_1 
               \end{pmatrix}   = 
      \begin{pmatrix}  1  & 0 \\ 0 & \varepsilon^{-3/2}\end{pmatrix}
      \begin{pmatrix}  r \\ s \end{pmatrix} ~,
\end{equation}
which puts forward $r(X,t)$ as the main component of the solution 
which is anticipated to be traveling principally to the right, with 
a relatively small scattered component $s_1(X,t)$ propagating 
principally to the left. The transformation leads to a modified
symplectic structure 
\begin{equation}
  J_3 = \frac{1}{\varepsilon^3} 
      \begin{pmatrix}  -\partial_X  & 
          \displaystyle{\frac{1}{4\varepsilon^{3/2}}
         \frac{\partial_X h_\varepsilon}{h_\varepsilon}} \\ 
        \displaystyle{ - \frac{1}{4\varepsilon^{3/2}}
         \frac{\partial_X h_\varepsilon}{h_\varepsilon}} &
          \displaystyle{\frac{1}{\varepsilon^{3}}} \partial_X \end{pmatrix}
\end{equation}
and a final Hamiltonian
\begin{equation}
\begin{array}{l}
 H_3^{\varepsilon}(r,s_1) = 
  \frac{\varepsilon^3}{2} \int \sqrt{gh_\varepsilon} (r^2+\varepsilon^3 s_1^2 )
    -c_1 \varepsilon^2  \Big((\partial_{X} r )^2 
    %\\
%\noalign{\vskip6pt}
 %\qquad\qquad 
 - 2 \varepsilon^{\frac32}(\partial_{X} r )
    (\partial_{X} s_1 )+\varepsilon^3 (\partial_{X} s_1)^2\Big) \\
\noalign{\vskip6pt}
   \qquad \qquad+ c_2\varepsilon^2  \Big(r^3 - \varepsilon^{\frac32} r^2 s_1 - 
     \varepsilon^3 r s_1^2 + \varepsilon^{\frac92}s_1^3\Big)  \, dX
      + \LittleOh{\varepsilon^5}.
\end{array}
\label{5.12}
\end{equation}
The equations stemming from the Hamiltonian (\ref{5.12}) and the 
above symplectic structure are 
\begin{equation}
\begin{array}{l}
 \partial_t r = -\partial_X \Big[ \sqrt{ g h_\varepsilon} r
     + \varepsilon^2 (c_1 \partial_X^2 r +\frac{3}{2} c_2 r^2) \\
\noalign{\vskip6pt}
  \qquad\qquad 
     + \varepsilon^2 (- \varepsilon^{\frac32} c_1 \partial_X^2 s_1 
       - \varepsilon^{\frac32} c_2 r s_1 - 
\frac{1}{2}\varepsilon^3 c_2 s_1^2)  
     \Big]        \\
\noalign{\vskip6pt}
  \quad - \frac{1}{4}\frac{ \partial_x \beta(\frac{X}{\varepsilon})}{h_\varepsilon} 
     \Big[ \varepsilon^{3/2} \sqrt{g h_\varepsilon} s_1 
   + \varepsilon^2 ( c_1 (- \partial_X^2 r 
     + \varepsilon^{3/2} \partial_X^2 s_1) \\
\noalign{\vskip6pt}  
\qquad\qquad   + c_2 ( - \frac{1}{2} r^2 - \varepsilon^{3/2} rs_1 
   + \frac{3}{2} \varepsilon^3 s_1^2 )) \Big]
\end{array}
\label{eqr1}
\end{equation}
\begin{equation}
\begin{array}{l}
 \partial_t s_1 = \partial_X \Big[ \sqrt{ g h_\varepsilon} s_1 
      + \varepsilon^2( c_1(-\varepsilon^{-\frac32}\partial_X^2 r +
 \partial_X^2 s_1)\\
\noalign{\vskip6pt}
  \quad\quad    + c_2( -\frac{1}{2} \varepsilon^{-\frac32} r^2 - rs_1 
      + \frac{3}{2} \varepsilon^{\frac32} s_1^2)) \Big]   \\
\noalign{\vskip6pt}
  \quad + \frac{1}{4}\frac{ \partial_x \beta(\frac{X}{\varepsilon})}{h_\varepsilon} 
      \Big[ \varepsilon^{-\frac32} \sqrt{g h_\varepsilon} r 
    + \varepsilon^{\frac12} ( c_1 (\partial_X^2 r - 
\varepsilon^{\frac32}\partial_X^2 s_1)\\
\noalign{\vskip6pt}
\qquad
   + c_2(\frac{3}{2} r^2 - \varepsilon^{\frac32} rs_1 
    - \frac{1}{2} \varepsilon^3 s_1^2)) \Big].
\end{array}
\label{eqs1}       
\end{equation}
It is ambiguous at this point precisely which terms of the above
system of partial differential equations play a r\^ole in the asymptotic
description of solutions in the limit as $\varepsilon$ tends to zero.
The transformation \eqref{Eq:InhomogeneousScaling} is not homogeneous 
in the perturbation parameter $\varepsilon$, and because of fluctuations 
there are numerous cancellations that occur in the remaining terms, 
not all of them having an influence on the asymptotic regime (see 
Lemmas \ref{Lemma3.2} and \ref{Lemma3.5} for example). 
We will show in the subsequent analysis of Section \ref{Section5.3}
 that the asymptotic 
behavior of solutions of equations (\ref{eqr1})(\ref{eqs1}) as 
$\varepsilon \to 0$ is governed by the following coupled 
system of equations, with an appropriate choice of the parameters
$a_{KdV}$ and $b$.
% which results from culling the system 
%\eqref{eqr1}\eqref{eqs1} of terms which only influence the
%solutions at higher order as governed by $\varepsilon$ :
\begin{eqnarray} \label{eqrnew}
\partial_t r &=& -\partial_X \Big[ c_\varepsilon(X) r
     +\varepsilon^2 (c_1 \partial_X^2 r +\frac{3}{2} c_2 r^2)\Big]
+ \varepsilon^2  b r 
\end{eqnarray}
%%%%     \Big]+  \frac{c_1}{64h^3} \varepsilon^2 {\sf E}((\partial_x\beta)^3) r
%% + \frac{1}{4}\varepsilon^{3/2}\partial_x 
%%   \beta(\frac{X}{\varepsilon}) \sqrt{\frac{g}{h}} s_1  
\begin{eqnarray}\label{eqsnew}
\partial_t s_1 &=&  \sqrt{g h} \partial_X s_1 
   + \frac{1}{4}\sqrt{\frac{g}{h}}\varepsilon^{-3/2} 
\partial_x \beta(\frac{X}{\varepsilon}) r,  
\end{eqnarray}
where the regularized velocity is 
$c_\varepsilon (X) = \sqrt{gh}(1-\frac{\varepsilon}{2h} \beta(X/\varepsilon)-
\varepsilon^2 a_{KdV})$. There are two free parameters in this system of 
equations, namely, $a_{KdV}$ and $b$. They will be determined by the
consistency analysis of Section \ref{Section5.3}
as fixed points of the solution process and the asymptotic analysis.
In the end we find that
\begin{eqnarray}
&&a_{KdV}= \frac{1}{2h}a_\beta +\frac{1}{4h^2} {\sf E}(\beta^2)
 + \frac{3c_1}{8h^2 \sqrt{gh}} {\sf E}( (\partial_x\beta)^2)
\\  
&&b= -\frac{7 c_1}{64h^3}  {\sf E}((\partial_x\beta)^3).
\end{eqnarray}

%The constant $b= \frac{c_1}{64h^3} \varepsilon^2 {\sf E}((\partial_x\beta)^3)$
%governs the stability of solutions, and the regularized velocity is
%$c_\varepsilon (X) = \sqrt{gh}(1-\varepsilon \beta(X/\varepsilon)-
%\varepsilon a_{KdV})$, where $a_{KdV} = *******$ and $b$ are determined in the 
%process of consistency analysis. 

%Recall that the  criterion (CR) as defined in Section \ref{homog}, is
%that a term $a(X,t; \varepsilon, \omega)$ is considered of order 
%$\BigOh{\varepsilon^r}$ if for any space-time test function 
%$\varphi(X,t) \in {\mathcal S}$ the measures ${\ls P}_\varepsilon$ 
%induced by 
%$\varepsilon^{-r} \int a(X,t; \varepsilon, \omega) \varphi(X,t) \,
%dXdt$ converge weakly to a limit ${\ls P}_0$ as $\varepsilon$
%tends to zero. 

%%%%%%%%%%%%%%%%%%%%%%%%%%%%%%%%%%%%%%%%%%%%%%%%%%%%%%
%%
%%
%%  Solution procedure for equations \eqref{eqsnew}
%%
%%
%%%%%%%%%%%%%%%%%%%%%%%%%%%%%%%%%%%%%%%%%%%%%%%%%%%%%%

\subsection{Solution procedure for the random KdV equations}

In this section we describe a reduction procedure for the system of equations 
\eqref{eqrnew}-\eqref{eqsnew} 
that expresses the solution component  $r(X,t)$ in terms of a solution 
$q(Y,\tau)$ of a  deterministic equation similar to the KdV equation,
under a random change of
variables $(Y \mapsto X(t,Y))$ and a scaling $\tau=\varepsilon^2 t$
to the KdV time.   The scattered component 
$s_1(X,t)$ is an expression involving integrations along
characteristics.
The solution depends upon the two parameters $a_{KdV}$ and $b$.
 We retain the regularized 
form of the characteristic velocity $c_\varepsilon(X)$, only taking the 
limit as $\varepsilon \to 0$ in expressions for the solution.
 
Substitute $r= \partial_X R$ into \eqref{eqrnew}; the resulting
equation for $R$ is
\begin{equation}
   \partial_t R = - c_\varepsilon(X)  \partial_X R - \varepsilon^2 
   (c_1 \partial^3_{X}R + \frac{3}{2} c_2 (\partial_X R)^2) +
\varepsilon^2  b R.
\end{equation}
Transform to characteristic coordinates as in Section 
\ref{Subsection:RandomCharacteristicCoordinates},
\begin{equation}\label{theflow}
   \frac{dX}{dt} =  c_\varepsilon(X) ~, \qquad 
   X(0)=Y ~.
\end{equation}
We denote the flow by  $X=\Phi_t^\varepsilon(Y)$, which is a regularized
realization dependent change of variables. 
Define $Q(Y,\tau)=R(X,t)$ so that $Q$ satisfies 
\begin{equation}\label{Eq:ClassicalKdV}
   \partial_\tau Q = - c_1 \partial_Y^3 Q - \frac{3}{2}
 c_2 (\partial_YQ)^2  + b Q ~.
\end{equation}
%After rescaling $t$ through $\tau = \varepsilon^2 t$, the classical 
%(and deterministic) KdV equation is satisfied by 
To solve the initial value problem, set $q(Y,0) = r(Y,0)=r^0(Y)$, and 
solve the deterministic  equation 
\begin{equation}
   \partial_\tau q = - c_1 \partial_Y^3 q - 3
 c_2 q \partial_Y q  + b q \label{classicalKdV}
\end{equation}
for $q(Y,\tau)= \partial_Y Q(Y,\tau)$. 
If $b=0$, equation \eqref{classicalKdV} is the classical
 KdV equation. Additionally, for each realization 
$\beta(x,\omega)$ the regularized ODE \eqref{theflow} defining the flow 
has a solution given by $X = X(t,Y; \varepsilon, \omega)$. With these
two ingredients, the solution $r(X,t)$ of  equation \eqref{eqrnew} is given by
\begin{equation}
        r(X,t) = \partial_X Q(Y(t,X; \varepsilon, \omega),\varepsilon^2 t) 
        = \partial_Y Q(Y(t,X; \varepsilon, \omega),\varepsilon^2 t)
     \partial_X Y(t,X; \varepsilon, \omega) \label{r-sol}
\end{equation}
where $\partial_Y X(t,Y; \varepsilon, \omega)$ is the Jacobian of the flow \eqref{theflow}
as described in section \ref{Subsection:RandomCharacteristicCoordinates}, 
and $\partial_X Y(t,X; \varepsilon, \omega)$ is its inverse. 
This is an expression of the solution of the regularized equation.

The equation \eqref{eqsnew} describes the scattered component of 
 the KdV system above, whose solution is expressed by
integration of a forcing term which is
given in terms of $r(X,t)$ along left-moving characteristics . Explicitly,
\begin{equation}
\begin{array}{lcl}
 s_1(X,t) &=& s_1^0(X + \sqrt{gh}t) \\
\noalign{\vskip6pt}
&&
 + \frac{\varepsilon^{-\frac32}}{4}\sqrt{\frac{g}{h}} 
  \int_0^t \partial_x \beta\big(\frac{X+\sqrt{gh}(t-t')}{\varepsilon}\big)
       r(X + \sqrt{gh}(t-t'),t') \, dt'
\\
& =& s_1^0(X + \sqrt{gh}t) +
 \frac{\varepsilon^{-\frac32}}{4h}  \int_X^{X + \sqrt{gh}t}
\partial_x \beta\big(\frac{\theta}{\varepsilon}\big)
r\big(\theta, t+ \frac{X-\theta}{\sqrt{gh}}\big) d\theta.
\end{array}
\label{s_1-solution}
\end{equation}
The small parameter $\varepsilon$ is still present in the regularization; 
to complete the description we consider the limit of the expressions 
 (\ref{r-sol})(\ref{s_1-solution}) as  $\varepsilon$ tends to zero. The solution of
 \eqref{Eq:ClassicalKdV} is smooth, and 
admits a Taylor expansion in its arguments.  The inverse Jacobian 
has an asymptotic expression as well. Therefore, one writes
\begin{equation}
\begin{array}{lcl}
  r(X,t) &=& \partial_X Q(Y(X,t;\omega), \varepsilon^2 t) =
 \partial_Y Q(Y(X,t;\omega),t)
     \partial_X Y(X,t;\omega)   \\
 \noalign{\vskip6pt}
&=& q(X - \sqrt{gh}t,\varepsilon^2  t)\Big( 1 +
     \frac{\varepsilon}{2h}
     (\beta(\frac{X}{\varepsilon}) - \beta(\frac{X-\sqrt{gh}t}{\varepsilon})
 \Big)  
 \\
%\noalign{\vskip6pt}
&\qquad& + \partial_X q(X - \sqrt{gh}t, \varepsilon^2 t)
 \frac{\varepsilon^2}{2h} 
     \int_{(X-\sqrt{gh}t)/\varepsilon}^{\frac{X}{\varepsilon}} \beta(t') \, dt' 
 +\cdots
\\
%\end{array}
%\nonumber
%\end{equation}
%\begin{equation}
%\begin{array}{lcl}
% r(X,t)
 \noalign{\vskip6pt}
 &=& q(X - \sqrt{gh} t, \varepsilon^2 t)  \\
&& + \partial_X \Big( q(X - \sqrt{gh}t, \varepsilon^2 t)
       \Bigl( \frac{\varepsilon^2}{2h} 
     \int_{(X-\sqrt{gh}t)/\varepsilon}^{\frac{X}{\varepsilon}} \beta(t') \, dt' 
    \Bigr)\Big) +  \BigOh{\varepsilon^2}.
\end{array}
\label{r-solution}
\end{equation}

\begin{Prop}
In the limit as $\varepsilon$ tends to zero, the expression 
\eqref{r-solution} for the solution of \eqref{eqrnew}
is asymptotic as a distribution to
\begin{equation}
\begin{array}{l}
   r(X,t) =  q(X - \sqrt{gh} t,\varepsilon^2 t) \\
\noalign{\vskip6pt}
\qquad\qquad
+ \frac{\varepsilon^{3/2}\sigma_\beta}{2h}\sqrt[4]{gh} 
 \partial_X \Big( q(X - \sqrt{gh}t,\varepsilon^2  t)
     B_{\omega(X)}(t) \Big)  +o(\varepsilon^{3/2}).
\end{array}
\label{r-limit}
\end{equation}
The expression for \eqref{s_1-solution} for the solution $s_1$ 
is asymptotic as a distribution to 
\begin{equation}
\begin{array}{lcl}
   &&  s_1(X,t) = s_1^0(X+\sqrt{gh}t)    \\
   \noalign{\vskip6pt}
   &&  + \frac{1}{4 h \sigma_\beta} \displaystyle{\int}_X^{X+\sqrt{gh}t} B(\theta)
         \frac{d^2}{d\theta^2} q(2\theta-X-\sqrt{gh}t,
         \varepsilon^2(t + \frac{X-\theta}{\sqrt{gh}})) d\theta \\
   \noalign{\vskip6pt}
   && + \frac{1}{4 h \sigma_\beta} 
      \Big( \partial_X B(X+\sqrt{gh}t) q(X+\sqrt{gh}t,0)
     -\partial_X B(X) q(X-\sqrt{gh}t,\varepsilon^2 t) \Big)\\
   \noalign{\vskip6pt}
   && -\frac{1}{2 h \sigma_\beta} 
     \Big( B(X+\sqrt{gh}t) \partial_X q(X+\sqrt{gh}t,0)
     -B(X) \partial_X q(X-\sqrt{gh}t,\varepsilon^2 t) \Big).
\end{array}
\label{s1-limit}
\end{equation}
\end{Prop}
\begin{proof}
The expression for the limit of $r_1$ follows directly from the
application of Lemma \ref{Lemma3.5}. It is  
an expression which exhibits both randomness in its amplitude, as well
as in location as per the random characteristic coordinates in which
it is expressed.  For the calculation for the limit of $s_1$, we substitute 
the expression (\ref{r-solution}) in (\ref{s_1-solution})  :
\begin{equation}
\begin{array}{l}
s_1(X,t) = s_1^0(X + \sqrt{gh}t)\\
\noalign{\vskip6pt}
\qquad+\displaystyle{\frac{\varepsilon^{-\frac32}}{4h}  \int}
_X^{X + \sqrt{gh}t}
\partial_x \beta\big(\frac{\theta}{\varepsilon}\big)
q\big(2\theta -X-\sqrt{gh}t,\varepsilon^2
( t+ \frac{X-\theta}{\sqrt{gh}})\big) d\theta\\
\noalign{\vskip6pt}
\qquad+ \displaystyle{\frac{\varepsilon^{1/2}}{8h^2}\int}_X^{X + \sqrt{gh}t}
\partial_x \beta\big(\frac{\theta}{\varepsilon}\big)
\Big[\int_{\frac{2\theta -X-\sqrt{gh}t}{\varepsilon}} ^{
\frac{\theta}{\varepsilon}}\beta(s)ds\\
\noalign{\vskip6pt}
\qquad\qquad\quad  
\times \partial_X q\big(2\theta -X-\sqrt{gh}t,\varepsilon^2(
 t+ \frac{X-\theta}{\sqrt{gh}})\big)
\Big]d\theta\\
\noalign{\vskip6pt}
\qquad + \displaystyle{\frac{\varepsilon^{-\frac12}}{8h^2}\int}_
X^{X + \sqrt{gh}t}
\Big(\partial_x \frac{\beta^2}{2}\big(\frac{\theta}{\varepsilon}\big)
-\partial_x \beta\big(\frac{\theta}{\varepsilon}\big)
\beta\big(\frac{2\theta-X-\sqrt{gh}t}{\varepsilon}\big)\Big)
\\
\noalign{\vskip6pt}
\qquad\qquad
%-\partial_x \beta\big(\frac{\theta}{\varepsilon}\big)
%\beta\big(\frac{2\theta-X-\sqrt{gh}t}{\varepsilon}\big)\Big)
\times q\big(2\theta -X-\sqrt{gh}t,\varepsilon^2( t+ \frac{X-\theta}{\sqrt{gh}})
\big) d\theta.
\end{array}
\label{s1-expression} 
\end{equation}
Except for the first term $s_1^0$ that remains unchanged, all the terms
appearing in the limiting expression \eqref{s1-limit} come from
the first integral in the expression of $s_1$, where we performed 
several integrations by parts and use the fact that $\partial_t 
q (X, \varepsilon^2 t) $ is 
$\BigOh{\varepsilon^2}$.
By more integration by parts, using the fact that 
$\partial_x = \varepsilon \partial_X$ we  can show  that the third term
 (third and fourth lines) in the expression
\eqref{s1-expression} is $\BigOh{\varepsilon^{1/2}}$.
Let us turn to the last term  (fifth and sixth lines) 
of \eqref{s1-expression}. For the term
containing $\partial_x \beta^2(\theta/\varepsilon)$, integration by parts
will produce an additional $\varepsilon$ and the term will eventually 
be of order 
$\BigOh{\varepsilon^{1/2}}$. To estimate the term containing the
product 
$\partial_x \beta\big(\frac{\theta}{\varepsilon}\big)
\beta\big(\frac{2\theta-X-\sqrt{gh}t}{\varepsilon}\big)$, the integration
by parts moves  
the derivative $\partial_x$ to all other terms. The only contribution
that will not produce an $\varepsilon$ is when the derivative
acts on  $\beta\big(\frac{2\theta-X-\sqrt{gh}t}{\varepsilon}\big)$.
For this term, we write
\begin{equation}
\partial_x\beta\big(\frac{2\theta-X-\sqrt{gh}t}{\varepsilon}\big)
= - \frac{\varepsilon}{\sqrt{gh}} 
\partial_t\beta\big(\frac{2\theta-X-\sqrt{gh}t}{\varepsilon}\big).
\end{equation}
After some simple manipulations, this term is again 
$\BigOh{\varepsilon^{1/2}}$.
\end{proof}
%%%%%%%%%%%%%%%%%%%%%%%%%%%%%%%%%%%%%%%%%%%%%%%%%%%
%%
%%
%%   Argument for the consistency of our resulting system of equations
%%
%%
%%%%%%%%%%%%%%%%%%%%%%%%%%%%%%%%%%%%%%%%%%%%%%%%%%%

\subsection{Consistency of the resulting system of equations}\label{Section5.3}
In this subsection, we complete the cycle of a self-consistency
analysis for equations \eqref{eqrnew} and \eqref{eqsnew}, out
of which the two so-far undetermined constants $a_{KdV}$ and
$b$ are selected. It is clear that not all terms in equations
\eqref{eqr1} and \eqref{eqs1} are of equal importance in the
limit as $\varepsilon \to 0$. Recall the  criterion 
as presented in section \ref{homog}, which states that a term
$a(X,t; \varepsilon, \omega)$ is of order
$\BigOh{\varepsilon^r}$ if for any space-time test function
$\varphi(X,t) \in {\mathcal S}$ the measures ${\sf P}_\varepsilon$
induced by
$\varepsilon^{-r}\int a(X,t;\varepsilon,\omega)\varphi(X,t)\, dXdt$
converge weakly to a limit ${\sf P}_0$ as $\varepsilon$ tends to zero.
In the present case, the  analysis consists of
(i) the derivation of  an expression for the solutions of 
 \eqref{eqrnew}\eqref{eqsnew} which  are stated in \eqref{s_1-solution} and
\eqref{r-solution},  and depend
upon the two parameters $a_{KdV}$ and $b$;
(ii) the examination  of the terms in \eqref{eqr1}, including in particular
those which do not appear in \eqref{eqrnew} (respectively, all the
terms in \eqref{eqs1}, in particular those that do not appear in
\eqref{eqsnew}). Using the expressions 
\eqref{s_1-solution}\eqref{r-solution}
we then show that, except terms which appear in 
\eqref{eqrnew} (respectively \eqref{eqsnew}), they are asymptotically of
order $o(\varepsilon^2)$ (respectively, of order $o(1)$).
 Both the system
\eqref{eqrnew}\eqref{eqsnew} and the solution expressions
\eqref{s_1-solution} \eqref{r-solution} depend upon 
parameters $a_{KdV}$ and $b$. (iii) The demonstration that these
constants can be chosen so that there is a fixed
point  of this analysis. Namely, the solution
depending upon the constants $a_{KdV}$ and $b$ has
asymptotic behavior which satisfies the equations
\eqref{eqrnew}-\eqref{eqsnew} with the same choice of
constants.

\smallskip\noindent
Let us denote the  terms in  (\ref{eqr1})  by
\begin{equation}
\begin{array}{lcl}
{\rm{I}}_r &=& \varepsilon^2\partial_X \Big(
 - \varepsilon^{3/2} c_1 \partial_X^2 s_1 
       - \varepsilon^{3/2} c_2 r s_1 - \frac{1}{2}\varepsilon^3 c_2 s_1^2 
     \Big)\\
\noalign{\vskip6pt}
{\rm{II}}_r &=&
-\frac{1}{4}\frac{\partial_x \beta(\frac{X}{\varepsilon})}{h_\varepsilon} 
      \varepsilon^{3/2} \sqrt{g h_\varepsilon} s_1 \\
\noalign{\vskip6pt}
{\rm{III}}_r&=& -\frac{1}{4}\frac{\partial_x \beta(\frac{X}{\varepsilon})}
{h_\varepsilon} 
    \varepsilon^2  c_1 (- \partial_X^2 r 
     + \varepsilon^{3/2} \partial_X^2 s_1)   \\
\noalign{\vskip6pt}
{\rm{IV}}_r&=&-
\frac{1}{4}\frac{ \partial_x \beta(\frac{X}{\varepsilon})}{h_\varepsilon} 
\varepsilon^2 c_2 ( - \frac{1}{2} r^2 - \varepsilon^{3/2} rs_1 
   + \frac{3}{2} \varepsilon^3 s_1^2 ) .
\end{array}
\end{equation}
Similarly, we denote the terms in  (\ref{eqs1}) 
by

\begin{equation}
\begin{array}{lcl} 
{\rm{I}}_s &=&\varepsilon^2
\partial_X \Big(
 c_1(-\varepsilon^{-3/2}\partial_X^2 r + \partial_X^2 s_1)
      + c_2( -\frac{1}{2} \varepsilon^{-3/2} r^2 - rs_1 
      + \frac{3}{2} \varepsilon^{3/2} s_1^2) \Big)\\
\noalign{\vskip6pt}
{\rm{II}}_s &=&
 \frac{1}{4}\frac{ \partial_x \beta(\frac{X}{\varepsilon})}{h_\varepsilon} 
 \varepsilon^{-3/2} \sqrt{g h_\varepsilon} r \\
\noalign{\vskip6pt}
{\rm{III}}_s &=&
 \frac{1}{4}\frac{ \partial_x \beta(\frac{X}{\varepsilon})}{h_\varepsilon} 
\varepsilon^{1/2} c_1 (\partial_X^2 r - \varepsilon^{3/2}\partial_X^2 s_1)
\\
\noalign{\vskip6pt}
{\rm{IV}}_s &=&
 \frac{1}{4}\frac{ \partial_x \beta(\frac{X}{\varepsilon})}{h_\varepsilon} 
\varepsilon^{1/2}
c_2(\frac{3}{2} r^2 - \varepsilon^{3/2} rs_1 
    - \frac{1}{2} \varepsilon^3 s_1^2).
\end{array}
\end{equation}
The purpose is to evaluate the asymptotic behavior of
each of these terms as $\varepsilon \to 0$.

\begin{Lemma}\label{Lemma5.3}
The term ${\rm{II}}_r$ has the asymptotic behavior
\begin{equation} 
{\rm{II}}_r =  \frac{1}{8h} \sqrt{\frac{g}{h}}\varepsilon^2    
{\sf E}(\beta^2) \partial_X r (X ,t) +o(\varepsilon^2).
\end{equation}
\end{Lemma}

\begin{Lemma}\label{Lemma5.4}
The term ${\rm{II}}_s$ has the behavior 
\begin{equation}
  {\rm{II}}_s =\frac{\varepsilon^{-3/2}}{4}\sqrt{\frac{g}{h}}
 \partial_x \beta(\frac{X}{\varepsilon}) r + o(1),  
\end{equation}
and this expression has an asymptotic limit as $\varepsilon \to 0$ which is
\begin{equation}
\frac{1}{4}\sqrt{\frac{g}{h}} \sigma_\beta \partial_X^2 B(X,\omega) 
q(X-\sqrt{gh}t,\tau).
% + \frac{1}{8h}\sqrt{\frac{g}{h}} {\sf E}(\beta^2)\partial_x
%q((X-\sqrt{gh}t,t). 
\label{asympIIs}
\end{equation}
\end{Lemma}

\begin{Lemma}\label{Lemma5.5}
\begin{equation}\label{IIIr}
{\rm{III}}_r = \frac{3c_1}{8h^2} \varepsilon^2 {\sf E}((\partial_x\beta)^2)
\partial_X r - \frac{7c_1}{64h^3} \varepsilon^2 {\sf E}((\partial_x\beta)^3) r
+o(\varepsilon^2),
\end{equation}
\begin{equation}
{\rm{III}}_s =  \varepsilon^{-3/2}  {\rm{III}}_r =\BigOh{\varepsilon^{1/2}}.
\end{equation}
\end{Lemma}

\begin{Lemma}\label{Lemma5.2}
The remaining terms have the following asymptotic behavior 
\begin{equation}
{\rm{I}}_r = o(\varepsilon^2 ) ~,  \qquad   {\rm{IV}}_r =o(\varepsilon^2)
\end{equation}
and 
\begin{equation}
{\rm{I}}_s = o(1) ~, \qquad {\rm{IV}}_s =o(1).
\end{equation}
\end{Lemma}

\begin{Lemma}\label{Lemma5.6}
Finally the linear term $-\partial_X(\sqrt{gh_\varepsilon} r) $
in the equation (\ref{eqrnew})  has the asymptotic behavior
\begin{equation}
-\partial_X(\sqrt{gh_\varepsilon} r)
= -\sqrt{gh} \partial_X\Big[ 
 \Big(1-\frac{\varepsilon}{2h} \beta(\frac{X}{\varepsilon})
- \frac{\varepsilon^2}{2h} (a_\beta 
 +\frac{1}{4h}
{\sf E}(\beta^2))  \Big)  r \Big] + o(\varepsilon^2).
\end{equation}  
\end{Lemma}

\smallskip\noindent
The proofs of these lemmas are the content of Section \ref{proofsoflemmas}. 
Using these asymptotic results in system \eqref{eqr1}\eqref{eqs1},
and retaining only the leading terms, it reduces to \eqref{eqrnew}\eqref{eqsnew}, 
with possibly different parameter values. When the parameters are
chosen appropriately, the asymptotic behavior of the equations matches that of
the solutions and the consistency procedure is closed.
\begin{Thm} \label{theorem5.7}
The result of the consistency analysis is that the free parameters in 
equations (\ref{eqrnew})(\ref{eqsnew}) are
\begin{eqnarray}
&&a_{KdV}= \frac{1}{2h}     a_\beta +\frac{1}{4h^2} {\sf E}(\beta^2) 
  + \frac{3c_1}{8h^2\sqrt{gh}} {\sf E}( (\partial_x\beta)^2),
\\  
&&b= -\frac{7c_1}{64h^3}  {\sf E}((\partial_x\beta)^3).
\end{eqnarray}
\end{Thm}
The parameter $a_{KdV}$ represents  an adjustment at $\BigOh{\varepsilon^2}$ to
the overall wavespeed, while the sign of $b$  governs the stability  
of solutions. In many cases, $b$ vanishes.

\begin{Prop} \label{Proposition5.7}
If the statistics of the ensemble $(\Omega,\mathcal{M},\sf{P})$ are reversible in $x$, then $b=0$.
\end{Prop}

By reversible, we  mean that the inversion $x\to -x$ preserves the probability 
measure ${\sf P}$, implying that
${\sf E}((\partial_x\beta)^3) =0$.

%%%%%%%%%%%%%%%%%%%%%%%%%%%
%%Section 5.4 Proofs of lemmas
%%%%%%%%%%%%%%%%%%%%%%%%%%%%%%
\subsection{Proofs of the above lemmas}\label{proofsoflemmas}

 In the analysis of the numerous integrals
that go in to this consistency result, it is convenient
to use the bracket notation as shorthand for integrations;
$$
  \langle f ,g\rangle := \int \int_{\Real^2} f(X,t) g(X,t) dXdt ~.
$$ 

\smallskip\noindent
%\begin{proof}
{\em Proof of Lemma \ref{Lemma5.3}:}
We first rewrite ${\rm{II}}_r$ as
\begin{equation}
{\rm{II}}_r= \frac{\sqrt{g}}{2} \varepsilon^{3/2}
\partial_X (\sqrt{h_\varepsilon}) s_1 = \frac{\sqrt{g}}{2} \varepsilon^{3/2}
\partial_X (\sqrt{h_\varepsilon} - {\sf E}(\sqrt{h_\varepsilon}))s_1.
\end{equation}
For any test function $\varphi(X,t)$, we compute 
$\langle \varphi, {\rm{II}}_r\rangle$ by 
substituting the expression (\ref{s_1-solution}) for $s_1$. This 
gives two terms, the first being
\begin{equation}
%\begin{array}{lcl}
 \varepsilon^{3/2}\frac{\sqrt{g}}{2}
\langle  \varphi, 
 \partial_X(\sqrt{h_\varepsilon} -{\sf E}(\sqrt{h_\varepsilon})) 
s_1^0\rangle 
%\noalign{\vskip6pt}
 = -  \frac{\sqrt{g}}{2}\varepsilon^{3/2}
\langle (\sqrt{h_\varepsilon} - {\sf E}(\sqrt{h_\varepsilon})) ,
 \partial_X(s_1^0\varphi) \rangle.
%\\
%%\noalign{\vskip6pt}
%&=&o(\varepsilon^2),
%\end{array}
%\nonumber
\end{equation}
Since
\begin{equation}
{\sf E}(\sqrt{h_\varepsilon}) = \sqrt{h}+\BigOh{\varepsilon^2},
\nonumber
\end{equation}
and because
\begin{equation}
\sqrt{h_\varepsilon}-{\sf E}(\sqrt{h_\varepsilon}) =
\sqrt{h}\big(1 -\frac{\varepsilon}{2h} \beta(\frac{X}{\varepsilon})\big) -\sqrt{h}
+ \BigOh{\varepsilon^2}  = 
-\frac{\varepsilon}{2\sqrt{h}} \beta(\frac{X}{\varepsilon})+
\BigOh{ \varepsilon^2},
%\nonumber
\end{equation}
the first term in $\langle \varphi, {\rm{II}}_r\rangle$ is of
order $o(\varepsilon^2)$.
The second term in the expression of $ \langle \varphi, {\rm{II}}_r\rangle$
is
 \begin{equation}
A:=\frac{\sqrt{g}}{8h}  \langle \varphi, \partial_X (\sqrt{h_\varepsilon}
-{\sf E}(\sqrt{h_\varepsilon}) )\int_X^{X + \sqrt{gh}t}
\partial_x \beta\big(\frac{\theta}{\varepsilon}\big)
r\big(\theta, t+ \frac{X-\theta}{\sqrt{gh}}\big) d\theta \rangle.
\end{equation}
By integration by parts,
\begin{equation}
\begin{array}{l}
A =
-\frac{\sqrt{g}}{8h}
~\langle\partial_X \varphi, (\sqrt{h_\varepsilon}
-{\sf E}(\sqrt{h_\varepsilon}) )
  \int_X^{X + \sqrt{gh}t}
\partial_x \beta\big(\frac{\theta}{\varepsilon}\big)
r\big(\theta, t+ \frac{X-\theta}{\sqrt{gh}}\big) d\theta \rangle
 \\
\noalign{\vskip7pt}
\qquad 
-\frac{\sqrt{g}}{8h}  \langle
\varphi,\big(\sqrt{h_\varepsilon} -{\sf E}(\sqrt{h_\varepsilon})\big)
\frac{1}{\sqrt{gh}}
\int_X^{X+\sqrt{gh}t}\partial_x \beta\big(\frac{\theta}{\varepsilon}\big)
\partial_t r\big(\theta, t+ \frac{X-\theta}{\sqrt{gh}}\big) d\theta
\rangle
 \\
\noalign{\vskip7pt}
\qquad-\frac{\sqrt{g}}{8h}  \langle
\varphi,\big(\sqrt{h_\varepsilon} -{\sf E}(\sqrt{h_\varepsilon})\big)
\Big[ \partial_x 
\beta(\frac{X+\sqrt{gh}t}{\varepsilon}) r^0(X+\sqrt{gh}t)
-\partial_x 
\beta(\frac{X}{\varepsilon}) r(X,t) \Big] \rangle
 \\
\noalign{\vskip7pt}
\quad= -\frac{\sqrt{g}}{8h}
~\langle(\partial_X - \frac{1}{\sqrt{gh}}\partial_t )
\varphi, (\sqrt{h_\varepsilon} -{\sf E}(\sqrt{h_\varepsilon}) )
  \int_X^{X + \sqrt{gh}t}
\partial_x \beta\big(\frac{\theta}{\varepsilon}\big)
r\big(\theta, t+ \frac{X-\theta}{\sqrt{gh}}\big) d\theta \rangle
 \\
\noalign{\vskip7pt}
 \qquad 
+\frac{\sqrt{g}}{8h}  ~\langle
\varphi,\big(\sqrt{h_\varepsilon} -{\sf E}(\sqrt{h_\varepsilon})\big)
 \partial_x 
\beta(\frac{X}{\varepsilon}) r(X,t) \rangle\\
\quad\equiv A_1+A_2 .
\end{array}
%\nonumber
\end{equation}
%By Lemma \ref{Lemma3.6}, $A_3= o(\varepsilon^2)$
Analyze the second term first,
\begin{equation}
A_2 =  - \frac{1}{32h} \sqrt{\frac{g}{h}}\varepsilon
\langle \varphi, \partial_x(\beta^2)~ r \rangle.
\end{equation}
Replacing $r$ by its expression (\ref{r-solution}),
\begin{equation}
\begin{array}{l}
A_2 = 
- \frac{1}{64h^2} \sqrt{\frac{g}{h}}\varepsilon
~\langle \varphi, \partial_x (\beta^2)
\Big[\varepsilon^2  \partial_X q 
\int_{\frac{X-\sqrt{gh}t}{\varepsilon}}^{\frac{X}{\varepsilon}}
 \beta(t') \, dt' 
\\
\noalign{\vskip6pt}
\qquad \qquad
+ \varepsilon q ~\big(\beta(\frac{X}{\varepsilon}) - 
\beta(\frac{X-\sqrt{gh}t}{\varepsilon}) \big)
\Big] \rangle +\BigOh{\varepsilon^{\frac52}}. 
\end{array}
\label{A2}
\end{equation}
By integration by parts, the first term of $A_2$ is $o(\varepsilon^2)$.
The second to the last term of $A_2$ can be rewritten 
as 
\begin{equation}
- \frac{1}{64h^2} \sqrt{\frac{g}{h}}\varepsilon^2 ~\langle \varphi,
\partial_x (\frac{2}{3}\beta^3) ~q\rangle
%-\frac{\sqrt{g}}{12h} \varepsilon^2
%~\langle \varphi q,
% \partial_x g_\varepsilon ~\rangle
\end{equation}
%where $ g_\varepsilon'(\alpha) = h_\varepsilon^{1/2} (\alpha)$ .
which  again by integration by parts  contributes to  $o(\varepsilon^2)$.
The last term of $A_2$  contributes only
$o(\varepsilon^2)$
 due to Lemma \ref{Lemma3.6}. Now turn to $A_1$.
\begin{equation}
\begin{array}{l}
A_1=-\frac{1}{16h} \sqrt{\frac{g}{h}}
 \varepsilon^2  \langle (\partial_X- \frac{1}{\sqrt{gh}}\partial_t) \varphi,  
\beta(\frac{X}{\varepsilon})
  \int_X^{X + \sqrt{gh}t}  \beta\big(\frac{\theta}{\varepsilon}\big)
\frac{d}{d\theta}r\big(\theta, t+ \frac{X-\theta}
{\sqrt{gh}}\big) d\theta \rangle 
\\
\noalign{\vskip6pt}
\qquad + \frac{1}{16h} \sqrt{\frac{g}{h}}\varepsilon^2  \langle
 (\partial_X- \frac{1}{\sqrt{gh}}\partial_t)\varphi,  
\beta(\frac{X}{\varepsilon})
\Big[\beta(\frac{X+\sqrt{gh}t}{\varepsilon}) r^0(X+\sqrt{gh}t)
-\beta(\frac{X}{\varepsilon}) r(X,t)\Big] \rangle\\
\noalign{\vskip6pt}
\qquad+o(\varepsilon^2). 
\end{array}
%\nonumber
\label{termA1}
\end{equation}
The first term in the second line of 
 $A_1$ is $o(\varepsilon^2)$ due to Lemma \ref{Lemma3.6}.
The last term of $A_1$ is 
\begin{equation}
\begin{array}{l}
-\frac{\varepsilon^2}{16h} \sqrt{\frac{g}{h}}{\sf E}(\beta^2)  ~\langle
(\partial_X- \frac{1}{\sqrt{gh}}\partial_t) \varphi,  
 q ~\rangle + o(\varepsilon^2) \\
\noalign{\vskip6pt}
\qquad = 
\frac{\varepsilon^2}{8h} \sqrt{\frac{g}{h}}{\sf E}(\beta^2)~\langle
\varphi,\partial_X q ~\rangle +
 o(\varepsilon^2),
\end{array}
\end{equation}
leading to a contribution to ${\rm II}_r$  of
\begin{equation}
\frac{\varepsilon^2  }{8h} \sqrt{\frac{g}{h}}  
{\sf E}(\beta^2) \partial_X r (X ,t) +
 o(\varepsilon^2).
\end{equation}
We now turn to the first term of $A_1$ which  we denote $A_3$, and write it 
as 
\begin{equation}
\begin{array}{l}
A_3= 
-\frac{\varepsilon^2}{16h} \sqrt{\frac{g}{h}}
  \langle (\partial_X- \frac{1}{\sqrt{gh}}\partial_t) \varphi,  
\beta(\frac{X}{\varepsilon})
  \int_X^{X + \sqrt{gh}t}  \beta\big(\frac{\theta}{\varepsilon}\big)
 (\partial_X- \frac{1}{\sqrt{gh}}\partial_t)  r 
(\theta, t+ \frac{X-\theta}
{\sqrt{gh}}\big) d\theta \rangle.  
\end{array}
\end{equation}
We express $(\partial_X- \frac{1}{\sqrt{gh}}\partial_t) r$ in terms of $q$
as 
\begin{equation} 
(\partial_X- \frac{1}{\sqrt{gh}}\partial_t) r (X,t) = 
2 \partial_X q + \frac{1}{2h}  \Big( \partial_x\beta(\frac{X}{\varepsilon})
- 2\partial_x\beta(\frac{X -\sqrt{gh}t}{\varepsilon}) \Big) q 
+\BigOh{\varepsilon}.
\end{equation}
Substitution of the above in $A_3$ gives rise to three terms,
$(i)$, $(ii)$, and $(iii)$ which have the form (after we have dropped  
 the constants):
\begin{equation}
\begin{array}{l}
(i) = \varepsilon^2\langle  (\partial_X- \frac{1}{\sqrt{gh}}\partial_t) \varphi,
\beta (\frac{X}{\varepsilon})
\int_X^{X+\sqrt{gh}t }
\beta (\frac{\theta}{\varepsilon}) \partial_Xq (2\theta-X-\sqrt{gh}t,
\varepsilon^2(  t+ \frac{X-\theta}{\sqrt{gh}})\big) d\theta \rangle 
\\
\noalign{\vskip9pt}
(ii) = \varepsilon^2
\langle  (\partial_X- \frac{1}{\sqrt{gh}}\partial_t) \varphi,
\beta (\frac{X}{\varepsilon})
\int_X^{X+\sqrt{gh}t }
\frac{1}{2}\partial_x (\beta^2 (\frac{\theta}{\varepsilon}))
 q (2\theta-X-\sqrt{gh}t,\varepsilon^2
   (t+ \frac{X-\theta}{\sqrt{gh}})\big) d\theta \rangle 
\\
\noalign{\vskip9pt}
(iii) =\varepsilon^2 
\langle 
(\partial_X- \frac{1}{\sqrt{gh}}\partial_t) \varphi,\\
\noalign{\vskip5pt}
\qquad\qquad\qquad\beta(\frac{X}{\varepsilon})
\int_X^{X+\sqrt{gh}t }
\beta (\frac{\theta}{\varepsilon})  \partial_x\beta
(\frac{2\theta -X-\sqrt{gh}t} {\varepsilon})
  q (2\theta-X-\sqrt{gh}t,\varepsilon^2
  ( t+ \frac{X-\theta}{\sqrt{gh}})) d\theta \rangle.
\end{array}
\nonumber
%\label{(iii)}
\end{equation}
The term $(i)$ is of the form 
\begin{equation} 
\langle \int_X^{X+\sqrt{gh}t } \beta(\frac{X}{\varepsilon})
\beta(\frac{\theta}{\varepsilon}) \psi(\theta,X,t) d\theta \rangle.
\end{equation}
Applying Lemma \ref{Lemma3.8}, we show that 
this term is $\BigOh{\varepsilon^3}$ , 
and thus does not contribute to the limit of  ${\rm II}_r$. 
By integration by parts, the term $(ii)$
is $\BigOh{\varepsilon^3}$. Finally, for  term $(iii)$, we write
$ \partial_x \beta( \frac{2\theta- X-\sqrt{gh}t}{\varepsilon})
=-\frac{\varepsilon}{\sqrt{gh}} \frac{d}{dt} 
\beta( \frac{2\theta- X -\sqrt{gh}t}{\varepsilon})$, 
leading to $(iii)$ being again of order $\BigOh{\varepsilon^3}$.

%\end{proof}
\qed

%\begin{proof}
\medskip\noindent
{\em Proof of Lemma \ref{Lemma5.4}:}
Using that $h_\varepsilon = h  -\varepsilon \beta(\frac{X}{\varepsilon}) 
+\BigOh{\varepsilon^2} $
\begin{equation}
{\rm{II}}_s =
 \frac{1}{4} \partial_x \beta(\frac{X}{\varepsilon})
\sqrt{\frac{g}{h}}\varepsilon^{-3/2}
 \Big(1+ \frac{\varepsilon}{2h} \beta (\frac{X}{\varepsilon})\Big) r.
\label{auxil1}
\end{equation}
Since  $r(X,t)= q(X-\sqrt{gh}t,\varepsilon^2 t) +\BigOh{\varepsilon}$,
the second term of (\ref{auxil1}) is
\begin{equation}
 \frac{\varepsilon^{-1/2}}{16h}\sqrt{\frac{g}{h}}
 \partial_x \beta^2(\frac{X}{\varepsilon})
 \Big(   q(X-\sqrt{gh}t,\varepsilon^2 t) +\BigOh{\varepsilon}\Big)
= \BigOh{\varepsilon^{1/2}}
\end{equation}
due to Lemma \ref{Lemma3.5}.
Compute the limit as $\varepsilon \to 0$ of 
${\rm{II}}_s$. Substituting the expression
(\ref{r-solution}) for $r$, we get, for any test function $\varphi(x,t)$
\begin{equation}
\begin{array}{lcl}
\langle \varphi,  \frac{\varepsilon^{-\frac32}}{4}\sqrt{\frac{g}{h}}
 \partial_x \beta(\frac{X}{\varepsilon}) r \rangle 
&=&\frac{\varepsilon^{-\frac32}}{4}\sqrt{\frac{g}{h}}~\langle \varphi,
\partial_x \beta(\frac{X}{\varepsilon})q\rangle \\
\noalign{\vskip6pt}
&& +\frac{\varepsilon^{-\frac12}}{4}\sqrt{\frac{g}{h}}~\langle \varphi,
\partial_x \beta(\frac{X}{\varepsilon}) \partial_X q
\frac{\varepsilon}{2h} \int_{\frac{X-\sqrt{gh}t}{\varepsilon}}^{
\frac{X}{\varepsilon}} 
\beta(t') \, dt' \rangle\\
\noalign{\vskip6pt}
&& +\frac{\varepsilon^{-\frac32}}{4}\sqrt{\frac{g}{h}}~\langle \varphi,
\partial_x \beta(\frac{X}{\varepsilon})\frac{\varepsilon}{2h} q 
\Big( \beta(\frac{X}{\varepsilon}) -\beta(\frac{X-\sqrt{gh}t}
{\varepsilon})  \Big)\rangle.
\end{array}
\label{asympauxIIs}
\end{equation}
The first term of the RHS of (\ref{asympauxIIs}) tends 
to the first term of (\ref{asympIIs}) by application of Lemma \ref{Lemma3.5}.
The second term of (\ref{asympauxIIs})
 is rewritten, by integration by parts, as
\begin{equation}
\begin{array}{l}
-\frac{\varepsilon^{\frac12}}{8h}\sqrt{\frac{g}{h}} ~\langle \partial_X(\varphi
\partial_X q),
\beta(\frac{X}{\varepsilon})
 \int_{(X-\sqrt{gh}t)/\varepsilon}^{
\frac{X}{\varepsilon}} 
\beta(t') \, dt' \rangle
- \frac{\varepsilon^{\frac12}}{8h}\sqrt{\frac{g}{h}} \langle
\varphi,\beta^2(\frac{X}{\varepsilon}) \partial_X q \rangle\\
\noalign{\vskip6pt}
\qquad+ \frac{\varepsilon^{\frac12}}{8h}\sqrt{\frac{g}{h}} \langle
\varphi,\beta(\frac{X}{\varepsilon}) \beta(\frac{X-\sqrt{gh}t}{\varepsilon})
 \partial_X q\rangle.
\end{array}
\nonumber
\end{equation}
Clearly  all terms are $o(1)$.
The third term of (\ref{asympauxIIs}) is rewritten 
\begin{equation}
\begin{array}{l}
\frac{\varepsilon^{-\frac12}}{16h}\sqrt{\frac{g}{h}} \langle \varphi,
\partial_x\beta^2(\frac{X}{\varepsilon}) q \rangle - 
\frac{\varepsilon^{-\frac12}}{8h}\sqrt{\frac{g}{h}} ~\langle \varphi,
\beta(\frac{X}{\varepsilon}) \partial_x \beta(
\frac{X-\sqrt{gh}t}{\varepsilon})  q\rangle,
\end{array}
%\nonumber
\end{equation}
which is $o(1)$ by application of Lemmas \ref{Lemma3.5} and \ref{Lemma3.6}.

%\end{proof}
\qed

\medskip\noindent
%\begin{proof}
{\em Proof of Lemma \ref{Lemma5.5}:}
Decompose ${\rm{III}}_r$ as the sum of the two terms 
\begin{eqnarray}
&&C=\frac{c_1}{4}\varepsilon^2 \frac{\partial_x \beta(\frac{X}{\varepsilon})}
{h_\varepsilon}   \partial_X^2 r ,\label{termCr}
\\
&&D=-\frac{c_1}{4} \varepsilon^{2+3/2}
\frac{\partial_x \beta(\frac{X}{\varepsilon})}
{h_\varepsilon}   \partial_X^2 s_1 .\label{termDr}
\end{eqnarray}
We compute $\partial_X^2 r $ from (\ref{r-solution}) and do not write terms
that will clearly  give a contribution  of $o(\varepsilon^2)$.
We get
\begin{equation}
\langle \varphi , C \rangle =    \frac{c_1}{8 h} \varepsilon^2
\langle  \varphi, \frac{\partial_x\beta}{h_\varepsilon} 
(3 \partial_X q\partial_x\beta   + \varepsilon^{-1}  q\partial_x^2 \beta  )
 \rangle  +o(\varepsilon^2). \label{termC}
\end{equation}
The first term in (\ref{termC}), denoted $C_1$  is
\begin{equation}
\langle \varphi ,C_1\rangle=  \frac{3c_1}{8 h^2} 
\varepsilon^2 {\sf E}((\partial_x\beta)^2)
\langle  \varphi, \partial_X q(X-\sqrt{gh}t,\varepsilon^2t) \rangle.
\end{equation}
We substitute $\partial_X q$ in terms of $r$ in $C_1$ using (\ref{r-solution})
and we write 
\begin{equation}
\partial_X q= \partial_X r - \frac{q}{2h}  \Big( \partial_x 
\beta(\frac{X}{\varepsilon}) - 
\partial_x \beta(\frac{X-\sqrt{gh}t}{\varepsilon})\Big) + \BigOh{\varepsilon}.
\end{equation}
We then conclude that $C_1$ can be written as a functional of $r$ as
\begin{equation}
\langle \varphi, C_1 \rangle =  \frac{3c_1}{8 h^2} \varepsilon^2 {\sf E}((\partial_x\beta)^2)
\langle  \varphi, \partial_X  r \rangle.
\end{equation}
The second term in $C$, denoted by $C_2$ is
\begin{eqnarray}
\langle \varphi, C_2 \rangle&=&   \frac{c_1}{16 h} \varepsilon
\langle  \varphi, \frac{1}{h_\varepsilon} \partial_x
(\partial_x\beta)^2  q
 \rangle  \nonumber\\
&=&  \frac{c_1}{16 h^2} \varepsilon \langle ~ \varphi, \partial_x
(\partial_x\beta)^2 (1+ \frac{\varepsilon}{h} \beta) ~q\rangle+
 \BigOh{\varepsilon^3} \nonumber \\
&=&  \frac{c_1}{16 h^3} \langle  \varphi,\varepsilon^2 \partial_x
(\partial_x\beta)^2 \beta q\rangle +
 o(\varepsilon^2) \nonumber \\
&=& - \frac{c_1}{16 h^3} \varepsilon^2 {\sf E}
\big((\partial_x\beta)^3\big)  \langle \varphi,q\rangle +
 o(\varepsilon^2).
\end{eqnarray}
We conclude that the term $C$ of ${\rm{III}}_r$ is 
\begin{equation}
C=  \frac{3c_1}{8 h^2} \varepsilon^2 {\sf E}((\partial_x\beta)^2)
\partial_X  r - \frac{c_1}{16 h^3} \varepsilon^2 {\sf E}
\big((\partial_x\beta)^3\big) r + o(\varepsilon^2).
\end{equation}
We compute the term $D$  of ${\rm{III}}_r$
given in (\ref{termDr}).
For this, we compute $\partial_X ^2 s_1$ in terms of $r$ and get:
\begin{equation}
\begin{array}{l}
\partial_X ^2 s_1(X,t) =\partial_X ^2 s_1^0(X+ \sqrt{gh}t)
 + \frac{1}{4h}\varepsilon^{-\frac32} \Big[
\frac{1}{\varepsilon} \partial_x^2 \beta(\frac{X+\sqrt{gh}t}{\varepsilon})
r^0(X+\sqrt{gh}t)  \\
\noalign{\vskip8pt}
\qquad+ \partial_x \beta(\frac{X+\sqrt{gh}t}{\varepsilon})
\partial_X r^0(X+\sqrt{gh}t)
- \frac{1}{\varepsilon} \partial_x^2 \beta(\frac{X}{\varepsilon}) r(X,t) 
- \partial_x \beta(\frac{X}{\varepsilon}) \partial_X r(X,t)  
 \\
\noalign{\vskip8pt}
\qquad+ \frac{1}{\sqrt{gh}} \partial_x \beta(\frac{X+\sqrt{gh}t}{\varepsilon})
 \partial_t r(X+\sqrt{gh}t,0) - \frac{1}{\sqrt{gh}}
 \partial_x \beta(\frac{X}{\varepsilon})\partial_t r(X,t)  \\
\noalign{\vskip8pt}
\qquad+\frac{1}{gh} \int_X^{X+\sqrt{gh}t}
\partial_x\beta\big(\frac{\theta}{\varepsilon}\big)
\partial_{tt} r\big(\theta, t+ \frac{X-\theta}{\sqrt{gh}}\big) d\theta
\Big].
\end{array}
\end{equation}
All terms containing the process $\beta$ or its derivatives at
two different points $X/\varepsilon$ and $(X+\sqrt{gh}t)/\varepsilon$
will not contribute because of Lemma \ref{Lemma3.6}. The term containing
$s_1^0$ will be $o(\varepsilon^2)$. The remaining terms that need attention
are
\begin{equation} \label{termD1}
\begin{array}{l}
%\langle \varphi, D\rangle =
\frac{c_1 \varepsilon^2}{16h}
  \langle \varphi,\frac{\partial_x\beta}{h_\varepsilon}
\Big[ \varepsilon^{-1} \partial_x^2 \beta ~r(X,t) + \partial_x\beta
( \partial_X r(X,t)+\frac{1}{\sqrt{gh}} \partial_t r(X,t)) \Big] \rangle
\\
\noalign{\vskip6pt}
\qquad - \frac{c_1 \varepsilon^2}{16} \frac{1}{gh^2}
  \langle \varphi,\frac{\partial_x\beta}{h_\varepsilon}
\int_X^{X+\sqrt{gh}t}
\partial_x\beta\big(\frac{\theta}{\varepsilon}\big)
\partial_{tt} r\big(\theta, t+ \frac{X-\theta}{\sqrt{gh}}\big) d\theta
\rangle.
\end{array}
\end{equation}
Noting that
 $ \partial_X r + \frac{1}{\sqrt{gh}} \partial_t r = \BigOh{\varepsilon}$,
we have that 
\begin{equation}
\frac{c_1 \varepsilon^2}{16h}
  \langle \varphi,\frac{\partial_x\beta}{h_\varepsilon}
\Big( \partial_x\beta
( \partial_X r(X,t)+\frac{1}{\sqrt{gh}} \partial_t r(X,t)) \Big) \rangle
= \BigOh{\varepsilon^3}.
\end{equation}
The first term in (\ref{termD1}) has the form
\begin{equation}\label{5.72}
\begin{array}{l}
\frac{c_1 \varepsilon}{32h^2}
  \langle \varphi, (1+ \frac{\varepsilon}{h} \beta(\frac{X}{\varepsilon}))
\partial_x( (\partial_x\beta)^2) r \rangle +o(\varepsilon^2)
 \\
\noalign{\vskip7pt}
\quad = \frac{c_1 \varepsilon}{32h^2}
  \langle \varphi, \partial_x( (\partial_x\beta)^2 -
{\sf E}((\partial_x\beta)^2) )r \rangle
- \frac{c_1 \varepsilon^2}{32h^3}
\langle \varphi,{\sf E}((\partial_x\beta)^3) r \rangle +o(\varepsilon^2).
\end{array}
\end{equation}
Integrating by parts the first term of \eqref{5.72}, we get two
 contributions; when the derivative acts on $\varphi$, it is 
$o(\varepsilon^2)$ using Lemma \ref{Lemma3.5} and the fact that 
$r= q+\BigOh{\varepsilon}$. When the derivative acts on $r$, we get:
\begin{equation}\label{termD2}
-\frac{c_1 \varepsilon^2}{32h^2}
  \langle ~\varphi,\big( (\partial_x\beta)^2 -
{\sf E}((\partial_x\beta)^2)\big)\partial_X  r ~\rangle.
\end{equation}
Here we replace $\partial_X r$ by its expression in terms of $q$:
\begin{equation}
\partial_X r= \partial_X q + \frac{1}{2h}  q \Big( \partial_x \beta
(\frac{X}{\varepsilon}) -\partial_x \beta(\frac{X-\sqrt{gh}t}{\varepsilon})
\Big) +\BigOh{\varepsilon}.
\end{equation}
The resulting contribution for (\ref{termD2}) is
\begin{equation} 
-\frac{c_1 \varepsilon^2}{32h^2}
  \langle ~\varphi,\big((\partial_x\beta)^2 -
{\sf E}((\partial_x\beta)^2)\big)\partial_X  r ~\rangle
=- \frac{c_1 \varepsilon^2}{64h^3}
\langle ~\varphi,{\sf E}((\partial_x\beta)^3) r ~\rangle +o(\varepsilon^2).
\end{equation}
The last term to consider is the fourth term of (\ref{termD1}) where
the derivatives with respect to $t$ can be moved outside the integral 
using the fact that
\begin{equation} \label{5.76}
\begin{array}{l}
\int_X^{X+\sqrt{gh}t}
\partial_x\beta\big(\frac{\theta}{\varepsilon}\big)
\partial_{tt} r\big(\theta, t+ \frac{X-\theta}{\sqrt{gh}}\big) d\theta
=\partial_{tt} \int_X^{X+\sqrt{gh}t}
\partial_x\beta\big(\frac{\theta}{\varepsilon}\big)
 r\big(\theta, t+ \frac{X-\theta}{\sqrt{gh}}\big) d\theta \\
\noalign{\vskip7pt}
\qquad-\sqrt{gh} \partial_x\beta(\frac{X+\sqrt{gh}t}{\varepsilon}) 
\partial_t r(X+\sqrt{gh}t,0)- 
gh\partial_x\beta(\frac{X+\sqrt{gh}t}{\varepsilon}) 
\partial_X r(X+\sqrt{gh}t,0)
\\
\noalign{\vskip7pt}
\qquad -
\frac{gh}{\varepsilon} \partial_{xx}\beta(\frac{X+\sqrt{gh}t}{\varepsilon})
 r(X+\sqrt{gh}t,0).
\end{array}
\end{equation}
Using Lemma \ref{Lemma3.6} again,
\begin{equation}
\begin{array}{l}
- \frac{c_1 \varepsilon^2}{16} \frac{1}{gh^2}
  \langle \varphi,\frac{\partial_x\beta}{h_\varepsilon}
\int_X^{X+\sqrt{gh}t}
\partial_x\beta\big(\frac{\theta}{\varepsilon}\big)
\partial_{tt} r\big(\theta, t+ \frac{X-\theta}{\sqrt{gh}}\big) d\theta\rangle
\\
\noalign{\vskip7pt}
\qquad
= - \frac{c_1 \varepsilon^2}{16} \frac{1}{gh^2}
  \langle \partial_{tt}\varphi,\frac{\partial_x\beta}{h_\varepsilon}
\int_X^{X+\sqrt{gh}t}
\partial_x\beta\big(\frac{\theta}{\varepsilon}\big)
 r\big(\theta, t+ \frac{X-\theta}{\sqrt{gh}}\big) d\theta\rangle
+o(\varepsilon^2).
\end{array}
\end{equation}
Using the derivative  in the first factor of $\partial_x\beta $
appearing in the above expression and integrating by parts leads
to the appearance of an additional $\varepsilon$, making the
expression  $\BigOh{\varepsilon^3}$.
We have obtained that 
\begin{equation} 
D = -\frac{3 c_1\varepsilon^2 }{64h^3}{\sf E}((\partial_x\beta)^3) r 
+o(\varepsilon^2).
\end{equation}
Adding the expression for $C$ and $D$ , we have shown that (\ref{IIIr}) 
describes the asymptotic behavior of  ${\rm {III_r}}$.
%\end{proof}
\qed

\medskip\noindent
{\em Proof of Lemma \ref{Lemma5.2}:} Following  the criterion of 
Section 3, these terms are integrated  against
test functions $\varphi$ , and derivatives can be moved to $\varphi$ by
integration by parts. 

\qed

\medskip\noindent
%\begin{proof}
{\it Proof of Lemma \ref{Lemma5.6}}. \
The regularized depth $h_\varepsilon$ is defined as
$h_\varepsilon(X) = h -\varepsilon \beta(\frac{X}{\varepsilon}) -\
\varepsilon ^2 a_\beta$. Thus the regularized linear wave speed is
\begin{equation}
\sqrt{g h_\varepsilon} = 
\sqrt{gh}\big(1- \frac{\varepsilon}{2h} \beta
-\varepsilon^2\frac{a_\beta}{2h} -\varepsilon^2\frac{\beta^2}{8h^2}\big)
+o(\varepsilon^2).
\end{equation}
The term $\langle \varphi, \partial_X(\beta^2 r)\rangle $ is
calculated as
\begin{equation}\label{linearterm}
\langle \varphi, \partial_X (\beta^2 r)\rangle =
\langle \varphi, {\sf E}(\beta^2)\partial_X r\rangle
- \langle \partial_X\varphi, (\beta^2- {\sf E}(\beta^2)) r\rangle.
\end{equation}
Using that $r =q + \BigOh{\varepsilon}$, we get that the second term in 
(\ref{linearterm}) is   $\BigOh{\sqrt{\varepsilon}}$.
%\end{proof}
\qed
%%%%%%%%%%%%%%%%%%%%%%%%%%%
%
%section: evolution equations for expectation of solutions
%
%%%%%%%%%%%%%%%%%%%%%%%%%%%%

\section{Remarks on the  expectation of solutions}
\setcounter{equation}{0}
%%%%%%%%%%%%%%%%%%%%%%%%%%%
It is normal to calculate ${\sf E}( r(X,t,\omega)) =p(X,t)$ as a basic
prediction of the solution $r(X,t,\omega)$ itself.  We remark
that $r(X,t,\omega)$ is  a realization dependent function where the randomness
manifests itself on the same level as dispersive and nonlinear effects.
In  the paper \cite{NS03} on {\it{apparent diffusion}}, the authors present an
 analysis of the function $p(X,t)$ in the case of the linear water wave
 problem with bottom given by $ \{ y=-h +\sqrt{\varepsilon}
 \beta(X/\varepsilon) \}$. 
In the fully nonlinear regime of the 
present paper, diffusion is weaker, and occurs only on time scales
larger than those of $\BigOh{1}$ in KdV time $\tau$, as the 
following calculation shows. 

In the sense of weak limits of probability measures,
as  $\varepsilon \to 0$, 
\begin{equation}
   r(X,t) = q(Y,\tau),
\end{equation}
where 
\begin{equation}
   Y = X-\sqrt{gh}t + \frac{\varepsilon^{3/2}}{2h}
      (gh)^{1/4} \sigma_\beta B(t)
    + \varepsilon^2 a_{KdV} \sqrt{gh} t,
       \quad {\rm{and}} \quad \tau=\varepsilon^2 t. 
\end{equation}
Compute the expectation of the main component of the solution $r$ : 
\begin{equation}
\begin{array}{l}
   {\sf E} (r(X,t)) =\displaystyle{\int_{-\infty}^{\infty}}
     q( X-\sqrt{gh}t + \frac{\varepsilon^{3/2}}{2h} \sigma_\beta (gh)^{1/4} u
     + \varepsilon^2 a_{KdV} \sqrt{gh} t, \tau) d\mu_{B(t)} (u)
     \\
     \noalign{\vskip7pt}
   \qquad = \displaystyle{
     \frac{1}{\sqrt{2\pi t }}\int_{-\infty}^\infty} 
     q( X-\sqrt{gh}t + \frac{\varepsilon^{3/2}}{2h}\sigma_\beta
     (gh)^{\frac14} u
   + \varepsilon^2 a_{KdV} \sqrt{gh}t, \tau) \text{e}^{-\frac{u^2}{2t}} du.
\end{array}
\end{equation}
Assuming that $\max_\tau  |q(.,\tau)|_{L^1} < \infty$, we have
for fixed t, 
\begin{eqnarray} 
   \max_X {\sf E} (r(X,t)) \le \max_{X'}
      \frac{1}{\sqrt{2\pi t }}\displaystyle{\int}_
     {-\infty}^{\infty} |q( X' 
   + \frac{\varepsilon^{3/2}}{2h} \sigma_\beta (gh)^{\frac14} u, \tau)| du    
            \nonumber \\
   \qquad \qquad \le  \frac{2h\varepsilon^{-\frac32}}{\sqrt{2\pi t}}
     (gh)^{-\frac14} \int_{-\infty}^{\infty} |q(v,\tau)| dv.
\end{eqnarray}
This time decay of order $\varepsilon^{-3/2} t^{-1/2} =
(\varepsilon \tau)^{-1/2} $ shows that the diffusion coefficient is
of order $\BigOh{\varepsilon}$, meaning that diffusion
effects 
occur at an order higher that the one considered for the
derivation of the KdV equation.
To observe diffusion created by random effect at the order of
the relevant terms for the KdV would require a scaling 
for the bottom variations of the form
$ -h+\sqrt{\varepsilon}\beta(x, \omega)$, which is a `rougher'
bottom that the one considered in this paper.
This is the natural scaling that was considered in the linear analysis
of  \cite{NS03}.
 However, such a hypothesis
also affects the nonlinear and dispersive nature of solutions and indeed it
will introduce additional terms in the nonlinear coupled system of equations
for $(r, s)$  that would have to be taken into account. 
This is beyond the scope of the present paper and is planned as the
focus of a subsequent study.

%%%%%%%%%%%%%%%%%%%%%%%%%%%%

%\ack 

\section*{Acknowledgments}
 WC would like to thank
S. R. S. Varadhan for his suggestions at the beginning of this
project.
WC has been partially supported by  the Canada Research Chairs Program and 
NSERC through grant number 238452-01, 
ODE  by a CRC postdoctoral fellowship, 
PG by the University of Delaware Research Foundation and NSF through grant
 number DMS-0625931,
and CS by NSERC through grant number 46179-05.

%%%%%%%%%%%%%%%%%%%%%%%%%%%
%%
%%
%%   bibliography 
%%
%%
%%%%%%%%%%%%%%%%%%%%%%%%%%%

%\section*{References}


\begin{thebibliography}{99}
\bibitem[1] {AL07}
Alvarez-Samaniego, B., Lannes, D.
\textit{Large time existence for 3D water waves and asymptotics},
 Preprint, 2007.

\bibitem[2]
%[Artiles \& Nachbin~(2004)]
{AN04} 
Artiles, W. and  Nachbin, A.,
\textit{Asymptotic nonlinear wave modeling through the
Dirichlet-to-Neumann operator.},
Methods Appl. Anal. {\textbf 11} (2004), no. 4, 475--492.  

%\bibitem[2]
%[Belzons, Guazzelli \& Parodi (1988)]
%{BGP88}
%Belzons, M., Guazzelli, E., and Parodi, O.
%\textit{Gravity waves on a rough bottom: experimental 
%evidence of one-dimensional localization}. 
%J. Fluid Mech. \textbf{186} pp. 539-558.
           
\bibitem[3] {B68}
Billingsley, P.,
\textit{Convergence of probability measures},
John Wiley \& Sons, Inc., New York-London-Sydney (1968).

\bibitem[4] {BCS02}
Bona, J., Chen, M.  \& Saut, J.-C.,  
\textit{Boussinesq equations and other systems for small-amplitude
long waves in nonlinear dispersive media.
I. Derivation and linear theory},
% II. The nonlinear theory}, 
J. Nonlinear Sci.  {\textbf{12}}  (2002),  283--318.
% {\textbf{17}}  (2004), 925--952. 

\bibitem[5] {BCL05}
Bona, J., Colin, T.,  \& Lannes, D.,  
\textit{Long wave approximations for water waves},
Arch. Ration. Mech. Anal.  {\textbf{ 17}}  (2005), 373--410.

\bibitem[6] {C07}
Chazel, F., 
\textit{Influence of bottom topography on long water waves}, 
Preprint, 2007.

\bibitem[7]{C85}
Craig, W.,
\textit{An existence theory for water waves and the Boussinesq and
          Korteweg-de Vries scaling limits},  
Comm. P. Diff. Eq. {\textbf {8}}(1985), pp. 787-1003.
                                                                              
\bibitem[8]
{CGK05}
Craig, W., Guyenne, P. and Kalisch, H.,
\textit{Hamiltonian long-wave expansions for free surfaces and
interfaces.} Comm. Pure Appl. Math. \textbf{58} (2005), 1587--1641.
 
\bibitem[9]
{CGNS05}
Craig, W., Guyenne, P., Nicholls, D. and Sulem, C.,
\textit{Hamiltonian long-wave expansions for water waves over a rough bottom.}
Proc. R. Soc. Lond. Ser. A Math. Phys. Eng. Sci. \textbf{461}
(2005),  no. 2055, 839--873.
 
%\bibitem[6]
%[Craig \& Nicholls (2002)]
%{CN02}
%Craig, W. and Nicholls, D.,
%\textit{Travelling two and three dimensional capillary gravity water waves.}
%SIAM J. Math. Anal. \textbf{32} (2000), no. 2, 323--359. 

\bibitem[10] {CS93} 
Craig, W. and Sulem, C. ,
\textit{Numerical simulation of gravity waves.}
J. Comput. Phys.  \textbf{108}  (1993),  no. 1, 73--83.

\bibitem[11]
{CSS92}
Craig, W., Sulem, C. and Sulem, P.-L.
\textit{Nonlinear modulation of gravity waves: a rigorous approach},
  Nonlinearity  {\textbf 5}  (1992),  no. 2, 497--522. 

\bibitem[12] {Cr67}
Cramer H. and Leadbetter M.R.,
\textit{Stationary and Related Stochastic Processes},
 John Wiley and Sons, Inc. New York - London - Sydney, 1967.

%\bibitem
%[Devillard \& Souillard (1986)]
%{DS86}
%Devillard, P. and Souillard, B.,
%\textit{ Polynomially decaying transmission for the nonlinear
% Schr\"odinger equation in a random medium},
% J. Stat. Phys. \textbf{3} (1986), 423--439

%\bibitem[Devillard, Dunlop \& Souillard (1988)]{DDS88}
%P. Devillard, F. Dunlop and B. Souillard.
%\textit{Localization of gravity waves on a channel with random bottom}.
% J. Fluid Mech. \textbf{186}  521--538.

\bibitem[13] {DK94}
Doukhan, P.
\textit{Mixing Properties and Examples},
Lecture Notes in Statistics {\textbf 85}, Springer--Verlag, 1994.

\bibitem[14] {GM03}
Grataloup, G. and Mei, C. C.,
\textit{Long waves in shallow water over a random seabed}, 
 Phys. Rev. E \textbf{68} (2003), 026314.

\bibitem[15] {H71}
Howe, M.S., 
\textit{On wave scattering by random inhomogeneities, with application
to the theory of weak bores},
 J. Fluid. Mech. \textbf{45} (1971), 785--804.

\bibitem[16] {KN86}
Kano, T. \& Nishida, T. 
\textit{A mathematical justification for Korteweg-de Vries equation
 and Boussinesq equation of water surface waves},   
Osaka J. Math.  \textbf{23}  (1986), 389--413.

\bibitem[17] {MH03}
Mei, C. C. and Hancock, M. 
\textit{Weakly nonlinear surface waves over a random seabed},
J. Fluid Mech. \textbf{475} (2003), 247--268. 

\bibitem[18] {ML04}
Mei, C. C. and  Li, Y.,
\textit{Evolution of solitons over a randomly rough seabed}, 
Phys. Rev. E (3) \textbf{70} (2004), no. 1, 016302.

\bibitem[19] {N95}
Nachbin, A.
{\textit{The localization length of randomly scattered water waves},
 J. Fluid Mech. \textbf {296} (1995), 353--372.

\bibitem[20] {N03}
Nachbin, A.,
\textit{A terrain-following Boussinesq system},
SIAM J. Appl. Math. {\textbf 63} (2003), no. 3, 905--922.

\bibitem[21]
{NS03}
Nachbin, A. and S\o lna, K.
\textit{Apparent diffusion due to topographic 
microstructure in shallow waters}, 
Phys. Fluids {\textbf 15} (2003), no. 1, 66--77.

\bibitem[22] {OY72}
Oodaira, H. and Yoshihara, K.
\textit{Functional central limit theorems for strictly stationary
              processes satisfying the strong mixing condition},
K\=odai Math. Sem. Rep. {\textbf 24} 1972, 259--269.

\bibitem[23] {PMH02} 
Pihl, J. H., Mei, C. C. and Hancock, M.
\textit{Surface gravity waves over a two-dimensional random seabed.}
Phys. Rev. E \textbf{66} (2002), 016611. 
 
\bibitem[24] {RP83}
Rosales, R. and  Papanicolaou, G.,
\textit{Gravity waves in a channel with a rough bottom.}
Stud. Appl. Math. \textbf{68} (1983), no. 2, 89--102.

\bibitem[25]{SW00}
Schneider, G. \& Wayne, C.E.,  
\textit{ The long-wave limit for the water wave problem. I. The case of zero
 surface tension},  Comm. Pure Appl. Math.  \textbf{53}  (2000), 1475--1535. 

\bibitem[26] {SP00}
S\o lna, K. and Papanicolaou, G.,
\textit{Ray theory for a locally layered random medium},
Waves Random Media \textbf{18} 2000, 151--198.

\bibitem[27] {ST69}
Strassen, V. and  Dudley, R. M.,
\textit{ The central limit theorem and $\varepsilon $-entropy},
 1969 Probability and Information Theory,  224--231.
 Lectures  Notes in  Mathematics, Springer, Berlin.

\bibitem[28] {W05}
Wright, J.D.,
\textit{ Corrections to the KdV approximation for water waves},
  SIAM J. Math. Anal.  {\textbf{37}}  (2005), 1161--1206. 

\bibitem[29] {Y83}
Yosihara, H., 
\textit{ Gravity waves on the free surface of an incompressible
 perfect fluid of finite depth},
  Publ. Res. Inst. Math. Sci. \textbf{ 18}  (1982), no. 1, 49--96.

\bibitem[30] {Z68}
Zakharov, V. E.,
\textit{Stability of periodic waves of finite amplitude
on the surface of a deep fluid.}
J. Appl. Mech. Tech. Phys. \textbf{9} (1968), 1990--1994.
} 
\end{thebibliography}
\end{document}